\documentclass[a4paper,reqno,10pt]{amsart}

\usepackage{amssymb}
\usepackage{amstext}
\usepackage{amsmath}
\usepackage{amscd}
\usepackage{amsthm}
\usepackage{amsfonts}
\usepackage{enumerate}
\usepackage{graphicx}
\usepackage{latexsym}
\usepackage{mathrsfs}
\usepackage[all,color,dvips]{xy}
\usepackage{mathtools}
\xyoption{all}

\usepackage{pstricks}
\usepackage{lscape}
\usepackage{comment}

\makeindex

\newtheorem{theorem}{Theorem}[section]
\newtheorem{proposition}[theorem]{Proposition}
\newtheorem{corollary}[theorem]{Corollary}
\newtheorem{lemma}[theorem]{Lemma}
\newtheorem{question}[theorem]{Question}

\theoremstyle{definition}
\newtheorem{example}[theorem]{Example}
\newtheorem{definition}[theorem]{Definition}
\newtheorem{definition-proposition}[theorem]{Definition-Proposition}

\newtheorem{remark}[theorem]{Remark}

\newcommand{\Hom}{\operatorname{Hom}\nolimits}
\renewcommand{\mod}{\mathsf{mod}\hspace{.01in}}
\newcommand{\proj}{\mathsf{proj}\hspace{.01in}}
\newcommand{\inj}{\mathsf{inj}\hspace{.01in}}

\newcommand{\xto}{\xrightarrow}
\newcommand{\im}{\operatorname{Im}\nolimits}
\newcommand{\pd}{\operatorname{pd}\nolimits}
\newcommand{\End}{\operatorname{End}\nolimits}
\newcommand{\Ext}{\operatorname{Ext}\nolimits}
\newcommand{\op}{\operatorname{op}\nolimits}
\newcommand{\ann}{\operatorname{ann}\nolimits}
\newcommand{\rad}{\operatorname{rad}\nolimits}
\newcommand{\Cok}{\operatorname{Cok}\nolimits}
\newcommand{\Ker}{\operatorname{Ker}\nolimits}
\newcommand{\Tr}{\operatorname{Tr}\nolimits}

\newcommand{\CC}{{\mathcal C}}
\newcommand{\TT}{{\mathcal T}}
\newcommand{\FF}{{\mathcal F}}

\newcommand{\functor}[1]{{\overline{#1}}}

\newcommand{\add}{\mathsf{add}\hspace{.01in}}
\newcommand{\Fac}{\mathsf{Fac}\hspace{.01in}}
\newcommand{\Sub}{\mathsf{Sub}\hspace{.01in}}
\newcommand{\KKb}{\mathsf{K}^{\rm b}}
\newcommand{\KKtwo}{\mathsf{K}^2}
\newcommand{\thick}{\mathsf{thick}\hspace{.01in}}
\newcommand{\Q}{{\rm Q}}
\newcommand{\G}{{\rm G}}
\newcommand{\ftors}{\mbox{\rm f-tors}\hspace{.01in}}
\newcommand{\sftors}{\mbox{\rm sf-tors}\hspace{.01in}}
\newcommand{\fftors}{\mbox{\rm ff-tors}\hspace{.01in}}
\newcommand{\ftorf}{\mbox{\rm f-torf}\hspace{.01in}}
\newcommand{\sftorf}{\mbox{\rm sf-torf}\hspace{.01in}}
\newcommand{\fftorf}{\mbox{\rm ff-torf}\hspace{.01in}}
\newcommand{\silt}{\mbox{\rm silt}\hspace{.01in}}

\newcommand{\twosilt}{\mbox{\rm 2-silt}\hspace{.01in}}
\newcommand{\twopresilt}{\mbox{\rm 2-presilt}\hspace{.01in}}
\newcommand{\ctilt}{\mbox{\rm c-tilt}\hspace{.01in}}
\newcommand{\rigid}{\mbox{\rm rigid}\hspace{.01in}}
\newcommand{\mrigid}{\mbox{\rm m-rigid}\hspace{.01in}}
\newcommand{\trigid}{\mbox{\rm $\tau$-rigid}\hspace{.01in}}
\newcommand{\sttilt}{\mbox{\rm s$\tau$-tilt}\hspace{.01in}}

\newcommand{\stcotilt}{\mbox{\rm s$\tau^-$-tilt}\hspace{.01in}}
\newcommand{\ttilt}{\mbox{\rm $\tau$-tilt}\hspace{.01in}}
\newcommand{\tcotilt}{\mbox{\rm $\tau^-$-tilt}\hspace{.01in}}
\newcommand{\tilt}{\mbox{\rm tilt}\hspace{.01in}}
\newcommand{\cotilt}{\mbox{\rm cotilt}\hspace{.01in}}
\newcommand{\iso}{\mbox{\rm iso}\hspace{.01in}}
\newcommand{\kD}{D}

\begin{document}
\title{$\tau$-tilting theory}
\author{Takahide Adachi, Osamu Iyama and Idun Reiten}
\address{Graduate School of Mathematics, Nagoya University, Chikusa-ku, Nagoya. 464-8602, Japan}
\email{m09002b@math.nagoya-u.ac.jp}
\address{Graduate School of Mathematics, Nagoya University, Chikusa-ku, Nagoya. 464-8602, Japan}
\email{iyama@math.nagoya-u.ac.jp}
\urladdr{http://www.math.nagoya-u.ac.jp/~iyama/}
\thanks{The second author was supported by JSPS Grant-in-Aid for Scientific Research 21740010, 21340003, 20244001 and 22224001.}
\address{Department of Mathematics, NTNU, Norway}
\email{idunr@math.ntnu.no}
\dedicatory{Dedicated to the memory of Dieter Happel}
\thanks{The second and third authors were supported by FRINAT grant 19660 from the Research Council of Norway.}

\begin{abstract}
The aim of this paper is to introduce $\tau$-tilting theory, which `completes' (classical) tilting theory from the viewpoint of mutation.
It is well-known in tilting theory that an almost complete tilting module for any finite dimensional algebra over a field $k$ is a direct summand of exactly 1 or 2 tilting modules.
An important property in cluster tilting theory is that an almost complete cluster-tilting object in a 2-CY triangulated category is a direct summand of exactly 2 cluster-tilting objects.
Reformulated for path algebras $kQ$, this says that an almost complete support tilting module has exactly two complements.
We generalize (support) tilting modules to what we call (support) $\tau$-tilting modules, and
show that an almost complete support $\tau$-tilting module has exactly two complements for any finite dimensional algebra.

For a finite dimensional $k$-algebra $\Lambda$, we establish bijections between functorially finite torsion classes in $\mod\Lambda$, 
support $\tau$-tilting modules and two-term silting complexes in $\KKb(\proj \Lambda)$.
Moreover these objects correspond bijectively to cluster-tilting objects in $\CC$
if $\Lambda$ is a 2-CY tilted algebra associated with a 2-CY triangulated category $\CC$.
As an application, we show that the property of having two complements holds also for two-term silting complexes in $\KKb(\proj \Lambda)$.
\end{abstract}
\maketitle
\tableofcontents

\section*{Introduction}
Let $\Lambda$ be a finite dimensional basic algebra over an algebraically closed field $k$,
$\mod\Lambda$\index{3mod@$\mod\Lambda$} the category of finitely generated left $\Lambda$-modules,
$\proj\Lambda$\index{3proj@$\proj\Lambda$} the category of finitely generated projective left $\Lambda$-modules and 
$\inj\Lambda$\index{3inj@$\inj\Lambda$} the category of finitely generated injective left $\Lambda$-modules.
For $M\in\mod\Lambda$, we denote by $\add M$\index{3add@$\add M$} (respectively, $\Fac M$\index{3fac@$\Fac M$}, $\Sub M$\index{3sub@$\Sub M$})
the category of all direct summands (respectively, factor modules, submodules) of finite direct sums of copies of $M$.
Tilting theory for $\Lambda$, and its predecessors, have been central in the
representation theory of finite dimensional algebras since the early
seventies \cite{BGP,APR,BB,HR,B}. When $T$ is
a (classical) tilting module (which always has the same number of
non-isomorphic indecomposable direct summands as $\Lambda$), there is  an
associated torsion pair $(\TT, \FF)$, where
$\TT=\Fac T$, and the interplay between tilting modules and
torsion pairs has played a central role. Another important fact is
that an almost complete tilting module $U$ can be completed in at most
two different ways to a tilting module \cite{RS,U}. Moreover there are exactly two ways if
and only if $U$ is a faithful $\Lambda$-module \cite{HU1}.

Even for a finite dimensional path algebra $kQ$, where $Q$  is a finite quiver with no oriented cycles, not all almost complete tilting modules $U$ are
faithful. However, for the associated cluster category $\CC_Q$, where we have cluster-tilting objects induced from
tilting modules over path algebras $kQ'$ derived
equivalent to  $kQ$, then the almost complete cluster-tilting objects
have exactly two complements \cite{BMRRT}. This fact, and its generalization to
2-Calabi-Yau triangulated categories \cite{IY}, plays an important
role in the categorification of cluster algebras. In the case of
cluster categories, this can be reformulated in terms of the path
algebra $\Lambda=kQ$ as follows \cite{IT,Ri}: A $\Lambda$-module
$T$ is {\it support tilting}\index{support tilting module} if $T$ is a tilting $(\Lambda/\langle e\rangle)$-module for some idempotent $e$ of $\Lambda$.
Using the more general class of support tilting modules,
it holds for path algebras that almost complete support tilting modules
can be completed in exactly two ways to support tilting modules.

The above result for path algebras does not necessarily hold for a finite dimensional algebra. 
The reason is that there may be sincere modules which are not faithful.
We are looking for a generalization  of tilting
modules where we have such a result, and where at the same time some of
the essential properties of tilting modules still hold.
It is then natural to try to find a class of modules satisfying the following properties:
\begin{itemize}
\item[(i)] There is a natural connection with torsion pairs in $\mod\Lambda$.
\item[(ii)] The modules have exactly $|\Lambda|$ non-isomorphic indecomposable direct summands,
where $|X|$\index{$|X|$} denotes the number of nonisomorphic indecomposable direct summands of $X$.
\item[(iii)] The analogs of basic almost complete tilting modules have exactly two complements.
\item[(iv)] In the hereditary case the class of modules should
  coincide with the classical tilting modules.
\end{itemize}
For the (classical) tilting modules we have in addition that when the almost complete ones
have two complements, then they are connected in a special short exact sequence.
Also there is a naturally associated quiver, where the isomorphism classes of tilting
modules are the vertices.

There is a generalization of classical tilting modules to tilting modules of finite projective dimension \cite{Ha,Miy}.
But it is easy to see that they do not satisfy the required properties.
The category $\mod\Lambda$ is naturally embedded in the derived category of
$\Lambda$. The tilting and silting complexes for $\Lambda$ \cite{Ri,AI,Ai}
are also extensions of the tilting modules.
An almost complete silting complex has infinitely many complements.
But as we shall see, things work well when we restrict to the two-term silting complexes.

In the module case, it turns out that a natural class of modules to consider is given as follows.
As usual, we denote by $\tau$ the AR translation (see section 1.2).

\begin{definition}\label{first definition}
\begin{itemize}
\item[(a)] We call $M$ in $\mod\Lambda$ {\it $\tau$-rigid}\index{taurigid module@$\tau$-rigid module} if $\Hom_{\Lambda}(M,\tau M)=0$.
\item[(b)] We call $M$ in $\mod\Lambda$ {\it $\tau$-tilting}\index{tautilting module@$\tau$-tilting module} (respectively, \emph{almost complete $\tau$-tilting}\index{almost complete tautilting module@almost complete $\tau$-tilting module})
if $M$ is $\tau$-rigid and $|M|=|\Lambda|$ (respectively, $|M|=|\Lambda|-1$).
\item[(c)] We call $M$ in $\mod\Lambda$ {\it support $\tau$-tilting}\index{support tautilting module@support $\tau$-tilting module} if there exists an idempotent $e$ of $\Lambda$ such that $M$ is a $\tau$-tilting $(\Lambda/\langle e\rangle)$-module.
\end{itemize}
\end{definition}

Any $\tau$-rigid module is rigid (i.e. $\Ext^1_\Lambda(M,M)=0$), and the converse holds if the projective dimension is at most one.
In particular, any partial tilting module is a $\tau$-rigid module, and any tilting module is a $\tau$-tilting module.
Thus we can regard $\tau$-tilting modules as a generalization of tilting modules.

The first main result of this paper is the following analog of Bongartz completion for tilting modules.

\begin{theorem}[Theorem \ref{thm2.3}]
Any $\tau$-rigid $\Lambda$-module is a direct summand of some $\tau$-tilting $\Lambda$-module.
\end{theorem}

As indicated above, in order to get our theory to work nicely, we need to consider support $\tau$-tilting modules.
It is often convenient to view them, and the $\tau$-rigid modules, as certain pairs of $\Lambda$-modules.

\begin{definition}\label{define support tau tilting}
Let $(M,P)$ be a pair with $M\in\mod\Lambda$ and $P\in\proj\Lambda$.
\begin{itemize}
\item[(a)] We call $(M,P)$ a {\it $\tau$-rigid} pair\index{taurigid pair@$\tau$-rigid pair} if $M$ is $\tau$-rigid and $\Hom_\Lambda(P,M)=0$.
\item[(b)] We call $(M,P)$ a {\it support $\tau$-tilting}\index{support tautilting pair@support $\tau$-tilting pair} 
(respectively, \emph{almost complete support $\tau$-tilting}\index{almost complete suppport tautilting pair@almost complete support $\tau$-tilting pair})
pair if $(M,P)$ is $\tau$-rigid and $|M|+|P|=|\Lambda|$ (respectively, $|M|+|P|=|\Lambda|-1$).
\end{itemize}
\end{definition}

These notions are compatible with those in Definition \ref{first definition} (see Proposition \ref{two notion of tau-rigidity} for details).
As usual, we say that $(M,P)$ is \emph{basic} \index{basic pair} if $M$ and $P$ are basic.
Similarly we say that $(M,P)$ is a \emph{direct summand} \index{direct summand of pair} of $(M',P')$ if $M$ is a direct summand of $M'$ and $P$ is a direct summand of $P'$.

The second main result of this paper is the following.

\begin{theorem}[Theorem \ref{2 complements}]\label{main theorem 1}
Let $\Lambda$ be a finite dimensional $k$-algebra.
Then any basic almost complete support $\tau$-tilting pair for $\Lambda$ is a direct summand of exactly two basic support $\tau$-tilting pairs.
\end{theorem}

These two support $\tau$-tilting pairs are said to be \emph{mutations}\index{mutation!support tautilting pair@support $\tau$-tilting pair} of each other.
We will define the support $\tau$-tilting quiver $\Q(\sttilt\Lambda)$ by using mutation (Definition \ref{support tau-tilting quiver}).

When extending (classical) tilting modules to tilting complexes or silting complexes we have
pointed out that we do not have exactly two complements in the almost complete case. But considering
instead only the two-term silting complexes, we prove that this is the case.

The third main result is to obtain a close connection between support
$\tau$-tilting modules and other important objects in tilting theory.
The corresponding definitions will be given in section 1.

\begin{theorem}[Theorems \ref{basic bijection}, \ref{sptsil4}, \ref{bijection between CT and PT} and \ref{ctsilt3}]\label{bijections}
Let $\Lambda$ be a finite dimensional $k$-algebra. We have bijections between
\begin{itemize}
\item[(a)] the set $\ftors\Lambda$\index{3ftors@$\ftors\Lambda$} of functorially finite torsion classes in $\mod\Lambda$,
\item[(b)] the set $\sttilt\Lambda$\index{3stautilt@$\sttilt\Lambda$} of isomorphism classes of basic support $\tau$-tilting modules,
\item[(c)] the set $\twosilt\Lambda$\index{2silt@$\twosilt\Lambda$} of isomorphism classes of basic two-term silting complexes for $\Lambda$,
\item[(d)] the set $\ctilt\CC$\index{3ctilt@$\ctilt\CC$} of isomorphism classes of basic cluster-tilting objects in a 2-CY triangulated category $\CC$ if $\Lambda$ is an associated 2-CY tilted algebra to $\CC$.
\end{itemize}
\end{theorem}

Note that the correspondence between (b) and (d) improves results in \cite{Smi,FL}.

By Theorem \ref{bijections}, we can regard $\sttilt\Lambda$ as a partially ordered set by using the inclusion relation of $\ftors\Lambda$ (i.e. we write $T\ge U$\index{$\ge$} if $\Fac T\supseteq\Fac U$).
Then we have the following fourth main result, which is an analog of \cite[Theorem 2.1]{HU2} and \cite[Theorem 2.35]{AI}.

\begin{theorem}[Corollary \ref{Hasse is mutation 2}]
The support $\tau$-tilting quiver $\Q(\sttilt\Lambda)$ is the Hasse quiver of the partially ordered set $\sttilt\Lambda$.
\end{theorem}

We have the following direct consequences of Theorem \ref{bijections}, where the second part is known by \cite{IY},
and the third one by \cite{ZZ}.

\begin{corollary}[Corollaries \ref{sptsil5}, \ref{application to CT}]
\begin{itemize}
\item[(a)] Two-term almost complete silting complexes have exactly two complements.
\item[(b)] In a 2-Calabi-Yau triangulated category with cluster-tilting objects, any almost complete cluster-tilting objects in a 2-CY category have exactly two complements.
\item[(c)] In a 2-Calabi-Yau triangulated category with cluster-tilting objects, any maximal rigid object is cluster-tilting.
\end{itemize}
\end{corollary}

Part (a) was first proved directly by Derksen-Fei \cite{DF} without dealing with support $\tau$-tilting modules.
Here we obtain this result by combining a bijection in Theorem \ref{bijections} with Theorem \ref{main theorem 1}.

Another important part of our work is to investigate to which extent the main properties of tilting
modules mentioned above remain valid in the settings of support $\tau$-tilting modules,
two-term silting complexes and cluster-tilting objects in 2-CY triangulated categories.

A motivation for considering the problem of exactly two complements for
almost complete support $\tau$-tilting modules was that the condition of $\tau$-rigid module
appears naturally when we express $\Ext^1_{\CC}(X,Y)$
for $X$ and $Y$ objects in a 2-CY category $\CC$ in terms of corresponding
modules $\overline{X}$ and $\overline{Y}$ over an associated 2-CY tilted
algebra (Proposition \ref{thm3.3}).

There is some relationship to the $E$-invariants of \cite{DWZ}
in the case of finite dimensional Jacobian algebras, where the expression
$\Hom_\Lambda(M,\tau N)$ appears. Here we introduce $E$-invariants in section 5 for any
finite dimensional $k$-algebras, and express them in terms of
dimension vectors and $g$-vectors as defined in \cite{DK}, inspired by \cite{DWZ}.

In the last section 6 we illustrate our results with examples.

There is a curious relationship with interesting independent work by Cerulli-Irelli,
Labardini-Fragoso and Schr\"oer \cite{CLS}, where the authors deal with $E$-invariants in
the more general setting of basic algebras which are not necessarily finite dimensional.
We refer to recent work by K\"onig and Yang \cite{KY} for connection with t-structures and co-t-structures.
Hoshino, Kato and Miyachi \cite{HKM} and Abe \cite{Ab} studied two term tilting complexes.
Buan and Marsh have considered a direct map from cluster-tilting objects
in cluster categories to functorially finite torsion classes for associated
cluster-tilted algebras.

\medskip
\noindent{\bf Acknowledgements }
Part of this work was done when the authors attended conferences
in Oberwolfach (February 2011), Banff (September 2011), Shanghai (October 2011)
and Trondheim (March 2012). 
Parts of the results in this paper were presented at conferences in Kagoshima (February 2012), Graz, Nagoya, Trondheim (March 2012), Matsumoto, 
Bristol (September 2012) and MSRI (October 2012).
The authors would like to thank the organizers of these conferences.
Part of this work was done while the second author visited Trondheim
in March 2010, March 2011 and March-April 2012. He would like to thank
the people at NTNU for hospitality and stimulating discussions.
We thank Julian K\"ulshammer and Xiaojin Zhang for pointing out typos in the first draft.

\section{Background and preliminary results}\label{sec1.1}
In this section we give some background material on each of the 4 topics
involved in our main results. This concerns the relationship
between tilting modules and functorially finite subcategories and some results
on $\tau$-rigid and $\tau$-tilting modules, including new
basic results about them which will be useful in the next section.
Further we recall known results on silting complexes, and on cluster-tilting
objects in 2-CY triangulated categories.

\subsection{Torsion pairs and tilting modules}\label{ssec1.1}
Let $\Lambda$ be a finite dimensional $k$-algebra.
For a subcategory $\CC$ of $\mod\Lambda$, we let
\begin{eqnarray*}
\CC^\perp&:=&\{X\in\mod\Lambda\mid \Hom_\Lambda(\CC,X)=0\},\\
\CC^{\perp_1}&:=&\{X\in\mod\Lambda\mid \Ext^1_\Lambda(\CC,X)=0\}.
\end{eqnarray*}\index{$(-)^\perp$, $(-)^{\perp_1}$, ${}^\perp(-)$, ${}^{\perp_1}(-)$}
Dually we define ${}^\perp\CC$ and ${}^{\perp_1}\CC$.
We call $T$ in $ \mod \Lambda$ a \emph{partial tilting module}\index{partial tilting module} 
if $\pd_{\Lambda}T\le 1$ and $\Ext^1_{\Lambda}(T,T)=0$.
A partial tilting module is called a \emph{tilting module}\index{tilting module}
if there is an exact sequence 
$0\to\Lambda\to T_0\to T_1\to 0$ with $T_0$ and $T_1$ in $\add T$.
Then any tilting module satisfies $|T|=|\Lambda|$.
Moreover it is known that for any partial tilting module $T$,
there is a tilting module $U$ such that $T\in\add U$ and $\Fac U=T^{\perp_1}$,
called the \emph{Bongartz completion}\index{Bongartz completion!tilting module} of $T$.
Hence a partial tilting module $T$ is a tilting module if and only if
$|T|=|\Lambda|$.
Dually $T$ in $\mod\Lambda$ is a (\emph{partial}) \emph{cotilting module}\index{partial cotilting module}\index{cotilting module} if $DT$ is
a (partial) tilting $\Lambda^{\op}$-module.

On the other hand, we say that a full subcategory $\TT$ of $\mod\Lambda$ is a \emph{torsion class}\index{torsion class} (respectively, \emph{torsionfree class}\index{torsionfree class}) if it is closed under factor modules (respectively, submodules) and extensions.
A pair $(\TT,\FF)$ is called a \emph{torsion pair}\index{torsion pair} if $\TT={}^\perp\FF$ and $\FF=\TT^\perp$.
In this case $\TT$ is a torsion class and $\FF$ is a torsionfree class. 
Conversely, any torsion class $\TT$ (respectively, torsionfree class $\FF$) gives rise to a torsion pair $(\TT,\FF)$.

We say that $X\in\TT$ is \emph{$\Ext$-projective}\index{extprojective@$\Ext$-projective module} (respectively, \emph{$\Ext$-injective}\index{extinjective@$\Ext$-injective module}) if $\Ext^1_\Lambda(X,\TT)=0$ (respectively, $\Ext^1_\Lambda(\TT,X)=0$).
We denote by $P(\TT)$\index{3p@$P(\TT)$} the direct sum of one copy of each of the indecomposable
$\Ext$-projective objects in $\TT$ up to isomorphism.
Similarly we denote by $I(\FF)$\index{3i@$I(\FF)$} the direct sum of one copy of each of the indecomposable $\Ext$-injective objects in $\FF$ up to isomorphism.

We first recall the following relevant result on torsion pairs and tilting modules.

\begin{proposition}\label{basics}\cite{AS,Ho,Sma}
Let $(\TT,\FF)$ be a torsion pair in $\mod\Lambda$. Then the following conditions are equivalent.
\begin{itemize}
\item[(a)] $\TT$ is functorially finite.
\item[(b)] $\FF$ is functorially finite.
\item[(c)] $\TT=\Fac X$ for some $X\text{ in } \mod \Lambda$.
\item[(d)] $\FF=\Sub Y$ for some $Y\text{ in }\mod \Lambda$.
\item[(e)] $P(\TT)$ is a tilting $(\Lambda/\ann\TT)$-module.
\item[(f)] $I(\FF)$ is a cotilting $(\Lambda/\ann\FF)$-module.
\item[(g)] $\TT=\Fac P(\TT)$.
\item[(h)] $\FF=\Sub I(\FF)$.
\end{itemize}
\end{proposition}

\begin{proof}
The conditions (a), (b), (c), (d), (e) and (f) are equivalent by \cite[Theorem]{Sma}.

(g)$\Rightarrow$(c) is clear.

(e)$\Rightarrow$(g) There exists an exact sequence $0\to\Lambda/\ann\TT\xto{a} T^0\to T^1\to0$ with $T^0,T^1\in\add P(\TT)$. For any $X\in\TT$, we take a surjection
$f:(\Lambda/\ann\TT)^\ell\to X$. It follows from $\Ext^1_\Lambda(T^1{}^\ell,X)=0$ that
$f$ factors through $a^\ell:(\Lambda/\ann\TT)^\ell\to T^0{}^\ell$.
Thus $X\in\Fac P(\TT)$.

Dually (h) is also equivalent to the other conditions.
\end{proof}

There is also a tilting quiver associated with the (classical) tilting modules. The vertices are the
isomorphism classes of basic tilting modules.
Let $X\oplus U$ and $Y\oplus U$ be basic tilting modules, where $X$ and $Y\ {\not \simeq}\ X$ are indecomposable.
Then it is known that there is some exact sequence $0\to X\xto{f} U'\xto{g} Y\to0$, where
$f:X\to U'$ is a minimal left $(\add U)$-approximation and $g:U'\to Y$ is a minimal right $(\add U)$-approximation. We say that $Y\oplus U$ is a left mutation of $X\oplus U$.
Then we draw an arrow $X\oplus U\to Y\oplus U$, so that we get a quiver for
the tilting modules. 
On the other hand, the set of basic tilting modules has a natural partial order given by $T\ge U$ if and only if $\Fac T\supseteq\Fac U$, and we can consider the associated Hasse quiver. These two quivers coincide \cite[Theorem 2.1]{HU2}.

\subsection{$\tau$-tilting modules}\label{ssec1.2}

Let $\Lambda$ be a finite dimensional $k$-algebra.
We have dualities
\[D:=\Hom_k(-,k):\mod\Lambda\leftrightarrow\mod\Lambda^{\op}\ \ \ \mbox{and}\ \ \ (-)^*:=\Hom_\Lambda(-,\Lambda):\proj\Lambda\leftrightarrow\proj\Lambda^{\op}\]\index{3d@$D$}\index{$(-)^{*}$}
which induce equivalences
\[\nu:=D(-)^*:\proj\Lambda\to\inj\Lambda\ \ \ \mbox{and}\ \ \ \nu^{-1}:=(-)^*D:\inj\Lambda\to\proj\Lambda\]\index{3nu@$\nu$,$\nu^{-1}$}
called \emph{Nakayama functors}\index{Nakayama functor}.
For $X$ in $\mod\Lambda$ with a minimal projective presentation
\[\xymatrix{P_1\ar[r]^{d_1}&P_0\ar[r]^{d_0}&X\ar[r]&0,}\]
we define $\Tr X$\index{3tr@$\Tr$} in $\mod\Lambda^{\op}$ and $\tau X$\index{3tau@$\tau$, $\tau^{-1}$} in $\mod\Lambda$ by exact sequences
\begin{eqnarray*}
\xymatrix{P_0^*\ar[r]^{d_1^*}&P_1^*\ar[r]^{}&\Tr X\ar[r]&0}\ \ \ \mbox{and}\ \ \ \xymatrix{0\ar[r]&\tau X\ar[r]&\nu P_0\ar[r]^{\nu d_1}&\nu P_1.}
\end{eqnarray*}
Then $\Tr$ and $\tau$ give bijections between the isomorphism classes of indecomposable non-projective $\Lambda$-modules,
the isomorphism classes of indecomposable non-projective $\Lambda^{\op}$-modules
and the isomorphism classes of indecomposable non-injective $\Lambda$-modules.
We denote by $\underline{\mod}\Lambda$\index{3mod@$\underline{\mod}\Lambda$, $\overline{\mod}\Lambda$} the \emph{stable category}\index{stable category} modulo projectives and by 
$\overline{\mod}\Lambda$\index{3mod@$\underline{\mod}\Lambda$, $\overline{\mod}\Lambda$} the \emph{costable category}\index{costable category} modulo injectives.
Then $\Tr$ gives the \emph{Auslander-Bridger transpose duality}\index{Auslander-Bridger transpose duality}
\[\Tr:\underline{\mod}\Lambda\leftrightarrow\underline{\mod}\Lambda^{\op}\]
and $\tau$ gives the \emph{AR translations}\index{Auslander-Reiten translation@AR translation}
\begin{eqnarray*}
\tau=D\Tr:\underline{\mod}\Lambda\to\overline{\mod}\Lambda\ \ \ \mbox{and}\ \ \ \tau^{-1}=\Tr D:\overline{\mod}\Lambda\to\underline{\mod}\Lambda.
\end{eqnarray*}\index{3tau@$\tau$, $\tau^{-1}$}
We have a functorial isomorphism
\[\underline{\Hom}_\Lambda(X,Y)\simeq D\Ext^1_\Lambda(Y,\tau X)\]
for any $X$ and $Y$ in $\mod\Lambda$ called \emph{AR duality}\index{Auslander-Reiten duality@AR duality}. In particular,
if $M$ is $\tau$-rigid, then we have $\Ext_\Lambda^1(M,M)=0$ (i.e. $M$ is rigid) by AR duality. More precisely, we have the following result, which we often use in this paper.

\begin{proposition}\label{Auslander-Smalo}
For $X$ and $Y$ in $\mod\Lambda$, we have the following.
\begin{itemize}
\item[(a)] \cite[Proposition 5.8]{AS} $\Hom_\Lambda(X,\tau Y)=0$ if and only if $\Ext_\Lambda^1(Y,\Fac X)=0$.
\item[(b)] \cite[Theorem 5.10]{AS} If $X$ is $\tau$-rigid, then $\Fac X$ is a functorially finite torsion class and $X\in\add P({\Fac X})$.
\item[(c)] If $\TT$ is a torsion class in $\mod\Lambda$, then $P(\TT)$ is a $\tau$-rigid $\Lambda$-module.
\end{itemize}
\end{proposition}

\begin{proof}
(c) Since $T:=P(\TT)$ is $\Ext$-projective in $\TT$, we have $\Ext_\Lambda^1(T,\Fac T)=0$.
This implies that $\Hom_\Lambda(T,\tau T)=0 $ by (a).
\end{proof}

We have the following direct consequence (see also \cite{Sk,ASS}).

\begin{proposition}\label{number of summands of tautilting}
Any $\tau$-rigid $\Lambda$-module $M$ satisfies $|M|\le|\Lambda|$.
\end{proposition}

\begin{proof}
By Proposition \ref{Auslander-Smalo}(b) we have $|M|\le|P(\Fac M)|$.
By Proposition \ref{basics}(e), we have $|P(\Fac M)|=|\Lambda/\ann M|$.
Since $|\Lambda/\ann M|\le|\Lambda|$, we have the assertion.
\end{proof}

As an immediate consequence, if $\tau$-rigid $\Lambda$-modules $M$ and $N$ satisfy $M\in\add N$ and $|M|\ge|\Lambda|$, then $\add M=\add N$.

Finally we note the following relationship between $\tau$-tilting modules and classical notions.

\begin{proposition}\cite[VIII.5.1]{ASS}\label{faithful}
\begin{itemize}
\item[(a)] Any faithful $\tau$-rigid $\Lambda$-module is a partial tilting $\Lambda$-module.
\item[(b)] Any faithful $\tau$-tilting $\Lambda$-module is a tilting $\Lambda$-module.
\end{itemize}
\end{proposition}

\subsection{Silting complexes}

Let $\Lambda$ be a finite dimensional $k$-algebra and $\KKb(\proj \Lambda)$\index{3kb@$\KKb(\proj \Lambda)$}
be the category of bounded complexes of finitely generated projective
$\Lambda$-modules.
We recall the definition of silting complexes and mutations. 

\begin{definition}\cite{AI,Ai,BRT,KV}
Let $P\in\KKb(\proj \Lambda)$.
\begin{enumerate}
\item[(a)] We call $P$ \emph{presilting}\index{presilting complex} if $\Hom_{\KKb(\proj \Lambda)}(P,P[i])=0$ for any $i > 0$.
\item[(b)] We call $P$ \emph{silting}\index{silting complex} if it is presilting and satisfies $\thick P=\KKb(\proj \Lambda)$, 
where $\thick P$ is the smallest full subcategory of $\KKb(\proj\Lambda)$
which contains $P$ and is closed under cones, $[\pm1]$, direct summands and isomorphisms.
\end{enumerate}
We denote by $\silt\Lambda$\index{3silt@$\silt\Lambda$} the set of isomorphism classes of basic silting complexes for $\Lambda$.
\end{definition}

The following result is important.

\begin{proposition}\cite[Theorem 2.27, Corollary 2.28]{AI}\label{number of summands of silting}
\begin{itemize}
\item[(a)] For any $P\in\silt\Lambda$, we have $|P|=|\Lambda|$.
\item[(b)] Let $P=\bigoplus_{i=1}^nP_n$ be a basic silting complex for $\Lambda$ with $P_i$ indecomposable. 
Then $P_1,\cdots,P_n$ give a basis of the Grothendieck group $K_0(\KKb(\proj\Lambda))$.
\end{itemize}
\end{proposition}

We call a presilting complex $P$ for $\Lambda$ \emph{almost complete silting}\index{almost complete silting complex} if $|P|=|\Lambda|-1$.
There is a similar type of mutation as for tilting modules.

\begin{definition-proposition}\cite[Theorem 2.31]{AI}\label{silt2.31}
Let $P=X\oplus Q$ be a basic silting complex with $X$ indecomposable.
We consider a triangle
\[
\xymatrix{
X\ar[r]^{f}&Q^{\prime}\ar[r]&Y\ar[r]&X[1]
}
\]
with a minimal left $(\add Q)$-approximation $f$ of $X$.
Then the \emph{left mutation}\index{left mutation!silting complex} of $P$ with respect to $X$ is $\mu^{-}_{X}(P):=Y\oplus Q$\index{3mu@$\mu^{+}$, $\mu^{-}$}.
Dually we define the \emph{right mutation}\index{right mutation!silting complex} $\mu^{+}_{X}(P)$\index{3mu@$\mu^{+}$, $\mu^{-}$} of $P$ with respect to $X$.\footnote{These notations $\mu^{-}$ and $\mu^{+}$ are the opposite of those in \cite{AI}. They are easy to remember since they are the same direction as $\tau^{-1}$ and $\tau$. }
Then the left mutation and the right mutation of $P$ are also basic silting complexes.
\end{definition-proposition}

There is the following partial order on the set $\silt\Lambda$.

\begin{definition-proposition}\cite[Theorem 2.11, Proposition 2.14]{AI}\label{silt2.11}
For $P,Q\in\silt\Lambda$, we write
\[P\geq Q\]
if $\Hom_{\KKb(\proj\Lambda)}(P,Q[i])=0$ for any $i>0$, which is equivalent to
$P^{\perp_{>0}}\supseteq Q^{\perp_{>0}}$ where $P^{\perp_{>0}}$ is a subcategory of $\KKb(\proj\Lambda)$ consisting of the $X$ satisfying $\Hom_{\KKb(\proj\Lambda)}(P,X[i])=0$ for any $i>0$.
Then we have a partial order on $\silt\Lambda$.
\end{definition-proposition}

We define the \emph{silting quiver}\index{silting quiver} $\Q(\silt\Lambda)$\index{3qsilt@$\Q(\silt\Lambda)$} of $\Lambda$ as follows:
\begin{itemize}
\item The set of vertices is $\silt\Lambda$.
\item We draw an arrow from $P$ to $Q$ if $Q$ is a left mutation of $P$.  
\end{itemize}
Then the silting quiver gives the Hasse quiver of the partially ordered set $\silt\Lambda$ by \cite[Theorem 2.35]{AI}, similar to the situation for tilting modules.
We shall later restrict to two-term silting complexes to get exactly two complements for almost complete silting complexes.

\subsection{Cluster-tilting objects}

Let $\CC$ be a $k$-linear Hom-finite Krull-Schmidt triangulated category. 
Assume that $\CC$ is \emph{2-Calabi-Yau}\index{2-Calabi-Yau category} (\emph{2-CY} for short) i.e. there exists a functorial isomorphism $D\Ext^1_{\CC}(X,Y)\simeq\Ext^1_{\CC}(Y,X)$.
An important class of objects in these categories are the cluster-tilting objects.
We recall the definition of these and related objects.

\begin{definition}
\begin{enumerate}
\item[(a)] We call $T$ in $\CC$ \emph{rigid}\index{rigid object} if $\Hom_{\CC}(T,T[1])=0$.
\item[(b)] We call $T$ in $\CC$ \emph{cluster-tilting}\index{cluster-tilting object} if $\add T =\{ X\in \CC \ |\ \Hom_{\CC}(T, X[1])=0 \}$.
\item[(c)] We call $T$ in $\CC$ \emph{maximal rigid}\index{maximal rigid object} if it is rigid and maximal with respect to this property, that is,
$\add T=\{ X\in \CC \ |\ \Hom_{\CC}(T\oplus X, (T\oplus X)[1])=0\}$.
\end{enumerate}
\end{definition}
We denote by $\ctilt\CC$\index{3ctilt@$\ctilt\CC$} the set of isomorphism classes of basic cluster-tilting objects in $\CC$.
In this setting, there are also mutations of cluster-tilting objects defined via approximations,
which we recall \cite{BMRRT,IY}.

\begin{definition-proposition}\label{ct3}\cite[Theorem 5.3]{IY}
Let $T=X\oplus U$ be a basic cluster-tilting object in $\CC$ and $X$ indecomposable in $\CC$.
We consider the triangle
\[
\xymatrix{
X\ar[r]^{f}&U'\ar[r]&Y\ar[r]&X[1]
}
\]
with a minimal left $(\add U)$-approximation $f$ of $X$.
Let $\mu^{-}_{X}(T):=Y\oplus U$\index{3mu@$\mu^{+}$, $\mu^{-}$}. Dually we define $\mu^{+}_{X}(T)$\index{3mu@$\mu^{+}$, $\mu^{-}$}.
A different feature in this case is that we have $\mu^{-}_{X}(T)\simeq\mu^{+}_{X}(T)$.
This is a basic cluster-tilting object which as before we call the \emph{mutation}\index{mutation!cluster-tilting object} of $T$ with respect to $X$.
\end{definition-proposition}

In this case we get just a graph rather than a quiver.
We define the \emph{cluster-tilting graph}\index{cluster-tilting graph} $\G(\ctilt\CC)$\index{3gctilt@$\G(\ctilt\CC)$} of $\CC$
as follows:
\begin{itemize}
\item The set of vertices is $\ctilt\CC$.
\item We draw an edge between $T$ and $U$ if $U$ is a mutation of $T$.
\end{itemize}
Note that $U$ is a mutation of $T$ if and only if $T$ and $U$ have all but one indecomposable direct summand in common \cite[Theorem 5.3]{IY} 
(see Corollary \ref{application to CT}(a)).

\section{Support $\tau$-tilting modules}\label{sec2}
Our aim in this section is to develop a basic theory of support $\tau$-tilting modules over any finite dimensional $k$-algebra.
We start with discussing some basic properties of $\tau$-rigid modules and connections between $\tau$-rigid
modules and functorially finite torsion classes (Theorem \ref{basic bijection}). 
As an application, we introduce Bongartz completion of $\tau$-rigid modules (Theorem \ref{thm2.3}).
Then we give characterizations of $\tau$-tilting modules (Theorem \ref{3 conditions}).
We also give left-right duality of $\tau$-rigid modules (Theorem \ref{left-right symmetry}).
Further we prove our main result which states that an almost complete support $\tau$-tilting module has exactly two complements 
(Theorem \ref{2 complements}). As an application, we introduce mutation of support $\tau$-tilting modules.
We show that mutation gives the Hasse quiver of the partially ordered set of support $\tau$-tilting modules (Theorem \ref{Hasse is mutation}).

\subsection{Basic properties of $\tau$-rigid modules}

When $T$ is a $\Lambda$-module with $I$ an ideal contained in $\ann T
$, we investigate the relationship between $T$ being $\tau$-rigid as a
$\Lambda$-module and as a $(\Lambda/I)$-module. We have the following.

\begin{lemma}\label{prop1.8}
Let $\Lambda$ be a finite dimensional algebra, and $I$ an ideal in
$\Lambda$. Let $M$ and $N$ be $(\Lambda/I)$-modules. Then we have the following.
\begin{itemize}
\item[(a)] If $\Hom_{\Lambda}(N,\tau M)=0$, then $\Hom_{\Lambda/I}(N,\tau_{\Lambda/I}M)=0$.
\item[(b)] Assume $I=\langle e\rangle$ for an idempotent $e$ in $\Lambda$.
Then $\Hom_{\Lambda}(N,\tau M)=0$ if and only if $\Hom_{\Lambda/I}(N,\tau_{\Lambda/I}M)=0$.
\end{itemize}
\end{lemma}

\begin{proof}
Note that we have a natural inclusion $\Ext^1_{\Lambda/I}(M,N)\to\Ext^1_{\Lambda}(M,N)$.
This is an isomorphism if $I=\langle e\rangle$ for an idempotent $e$ since $\mod(\Lambda/\langle e\rangle)$ is closed under extensions in $\mod\Lambda$.

(a) Assume $\Hom_{\Lambda}(N,\tau M)=0$.
Then by Proposition \ref{Auslander-Smalo}, we have $\Ext^1_{\Lambda}(M,\Fac N)=0$.
By the above observation, we have $\Ext^1_{\Lambda/I}(M,\Fac N)=0$.
By Proposition \ref{Auslander-Smalo} again, we have $\Hom_{\Lambda/I}(N,\tau_{\Lambda/I}M)=0$.

(b) Assume that $I=\langle e\rangle$ and $\Hom_{\Lambda/I}(N,\tau_{\Lambda/I}M)=0$.
By Proposition \ref{Auslander-Smalo}, we have $\Ext^1_{\Lambda/I}(M,\Fac N)=0$. By the above observation, we have $\Ext^1_{\Lambda}(M,\Fac N)=0$.
By Proposition \ref{Auslander-Smalo} again, we have $\Hom_{\Lambda}(N,\tau M)=0$.
\end{proof}

Recall that $M$ in $\mod\Lambda$ is \emph{sincere}\index{sincere module} if every simple $\Lambda$-module appears as a composition factor in $M$.
This is equivalent to the fact that there does not exist a non-zero idempotent $e$ of $\Lambda$ which annihilates $M$.

\begin{proposition}\label{sincere}
\begin{itemize}
\item[(a)] $\tau$-tilting modules are precisely sincere support $\tau$-tilting modules.
\item[(b)] Tilting modules are precisely faithful support $\tau$-tilting modules.
\item[(c)] Any $\tau$-tilting (respectively, $\tau$-rigid) $\Lambda$-module $T$ is a tilting (respectively, partial tilting) $(\Lambda/\ann T)$-module.
\end{itemize}
\end{proposition}

\begin{proof}
(a) Clearly sincere support $\tau$-tilting modules are $\tau$-tilting.
Conversely, if a $\tau$-tilting $\Lambda$-module $T$ is not sincere, then there exists a non-zero idempotent $e$ of $\Lambda$ such that $T$ is a $(\Lambda/\langle e\rangle)$-module.
Since $T$ is $\tau$-rigid as a $(\Lambda/\langle e\rangle)$-module by Lemma \ref{prop1.8}(a),
we have $|T|=|\Lambda|>|\Lambda/\langle e\rangle|$, a contradiction to Proposition \ref{number of summands of tautilting}.

(b) Clearly tilting modules are faithful $\tau$-tilting.
Conversely, any faithful support $\tau$-tilting module $T$ is partial tilting by Proposition \ref{faithful} and satisfies $|T|=|\Lambda|$.
Thus $T$ is tilting.

(c) By Lemma \ref{prop1.8}(a), we know that $T$ is a faithful $\tau$-tilting (respectively, $\tau$-rigid) $(\Lambda/\ann T)$-module.
Thus the assertion follows from (b) (respectively, Proposition \ref{faithful}).
\end{proof}

Immediately we have the following basic observation, which will be used frequently in this paper.

\begin{proposition}\label{two notion of tau-rigidity}
Let $(M,P)$ be a pair with $M\in\mod\Lambda$ and $P\in\proj\Lambda$. 
Let $e$ be an idempotent of $\Lambda$ such that $\add P=\add\Lambda e$.
\begin{itemize}
\item[(a)] $(M,P)$ is a $\tau$-rigid (respectively, support $\tau$-tilting, almost complete support $\tau$-tilting) pair for $\Lambda$ if and only if
$M$ is a $\tau$-rigid (respectively, $\tau$-tilting, almost complete $\tau$-tilting) $(\Lambda/\langle e\rangle)$-module.
\item[(b)] If $(M,P)$ and $(M,Q)$ are support $\tau$-tilting pairs for $\Lambda$, then $\add P=\add Q$. In other words, $M$ determines $P$ and $e$ uniquely.
\end{itemize}
\end{proposition}

\begin{proof}
(a) The assertions follow from Lemma \ref{prop1.8} and the equation $|\Lambda/\langle e\rangle|=|\Lambda|-|P|$.

(b) This is a consequence of Proposition \ref{sincere}(a).
\end{proof}

The following observations are useful.

\begin{proposition}\label{no common summand}
Let $X$ be in $\mod\Lambda$ with a minimal projective presentation $P_1\xto{d_1}P_0\xto{d_0} X\to0$.
\begin{itemize}
\item[(a)] For $Y$ in $\mod\Lambda$, we have an exact sequence
\[0\to\Hom_\Lambda(Y,\tau X)\to D\Hom_\Lambda(P_1,Y)\xto{D(d_1,Y)}D\Hom_\Lambda(P_0,Y)\xto{D(d_0,Y)}D\Hom_\Lambda(X,Y)\to0.\]
\item[(b)] $\Hom_\Lambda(Y,\tau X)=0$ if and only if the map $\Hom_\Lambda(P_0,Y)\xto{(d_1,Y)}\Hom_\Lambda(P_1,Y)$ is surjective.
\item[(c)] $X$ is $\tau$-rigid if and only if the map $\Hom_\Lambda(P_0,X)\xto{(d_1,X)}\Hom_\Lambda(P_1,X)$ is surjective.
\end{itemize}
\end{proposition}

\begin{proof}
(a) We have an exact sequence $0\to\tau X\to\nu P_1\xrightarrow{\nu d_1} \nu P_0$. Applying $\Hom_\Lambda(Y,-)$, we have a commutative diagram of exact sequences:
\[\xymatrix@C=3em{
0\longrightarrow\Hom_\Lambda(Y,\tau X)\ar[r]&\Hom_\Lambda(Y,\nu P_1)\ar[r]^{(Y,\nu d_1)}\ar[d]^\wr&\Hom_\Lambda(Y,\nu P_0)\ar[d]^\wr\\
&D\Hom_\Lambda(P_1,Y)\ar[r]^{D(d_1,Y)}&D\Hom_\Lambda(P_0,Y)\ar[r]^(.43){D(d_0,Y)}&D\Hom_\Lambda(X,Y)\longrightarrow0.
}\]
Thus the assertion follows.

(b)(c) Immediate from (a).
\end{proof}

We have the following standard observation (cf. \cite{HU2,DK}).

\begin{proposition}\label{no common summand 2}
Let $X$ be in $\mod\Lambda$ with a minimal projective presentation $P_1\xto{d_1}P_0\xto{d_0} X\to0$.
If $X$ is $\tau$-rigid, then $P_0$ and $P_1$ have no non-zero direct summands in common.
\end{proposition}

\begin{proof}
We only have to show that any morphism $s:P_1\to P_0$ is in the radical.
By Proposition \ref{no common summand}(c), there exists $t:P_0\to X$ such that $d_0s=td_1$.
Since $P_0$ is projective, there exists $u:P_0\to P_0$ such that $t=d_0u$.
Since $d_0(s-ud_1)=0$, there exists $v:P_1\to P_1$ such that $s=ud_1+d_1v$.
\[\xymatrix{
&P_1\ar[r]^{d_1}\ar[d]^s\ar@{.>}[dl]_v&P_0\ar[r]^{d_0}\ar[d]^t\ar@{.>}[dl]^u&X\ar[r]&0\\
P_1\ar[r]_{d_1}&P_0\ar[r]_{d_0}&X\ar[r]&0
}\]
Since $d_1$ is in the radical, so is $s$. Thus the assertion is shown.
\end{proof}

The following analog of Wakamatsu's lemma \cite{AR4} will be useful.

\begin{lemma}\label{prop1.9}
Let $\eta:0\to Y\to T'\xto{f}X$ be an exact sequence in $\mod\Lambda$, where
$T$ is $\tau$-rigid, and $f:T'\to X$ is a right $(\add T)$-approximation. Then we have $Y\in{}^{\perp}(\tau T)$.
\end{lemma}

\begin{proof}
Replacing $X$ by $\im f$, we can assume that $f$ is surjective.
We apply $\Hom_{\Lambda}(-,\tau T)$ to $\eta$ to get the exact sequence 
\[0=\Hom_{\Lambda}(T',\tau T)\to \Hom_{\Lambda}(Y,\tau T)\to\Ext_{\Lambda}^1(X,\tau T)\xto{\Ext^1(f,\tau T)}\Ext_{\Lambda}^1(T',\tau T),\]
where we have $\Hom_{\Lambda}(T',\tau T)=0$ because $T$ is $\tau$-rigid.
Since $f:T'\to X$ is a right $(\add T)$-approximation, the induced map
$(T,f):\Hom_{\Lambda}(T,T')\to\Hom_{\Lambda}(T,X)$ is surjective.
Then also the induced map
$\underline{\Hom}_{\Lambda}(T,T')\to\underline{\Hom}_{\Lambda}(T,X)$ of the maps modulo
projectives is surjective, so by the AR duality the map
$\Ext^1(f,\tau T):\Ext^1_{\Lambda}(X,\tau T)\to \Ext^1_{\Lambda}(T',\tau T)$ is injective.
It follows that $\Hom_{\Lambda}(Y,\tau T)=0$.
\end{proof}

\subsection{$\tau$-rigid modules and torsion classes}\label{ssec2.1}

The following correspondence is basic in our paper, where we denote by
$\ftors\Lambda$\index{3ftors@$\ftors\Lambda$} the set of functorially finite torsion classes in $\mod\Lambda$.

\begin{theorem}\label{basic bijection}
There is a bijection
\[\sttilt\Lambda\longleftrightarrow\ftors\Lambda\]
given by $\sttilt\Lambda\ni T\mapsto\Fac T\in\ftors\Lambda$ and $\ftors\Lambda\ni\TT\mapsto P(\TT)\in\sttilt\Lambda$.
\end{theorem}

\begin{proof}
Let first $\TT$ be a functorially finite torsion class in $\mod\Lambda$.
Then we know that $T=P(\TT)$ is $\tau$-rigid by Proposition \ref{Auslander-Smalo}(c).
Let $e\in\Lambda$ be a maximal idempotent such that
$\TT\subseteq \mod(\Lambda/\langle e\rangle)$.
Then we have $|\Lambda/\langle e\rangle|=|\Lambda/\ann\TT|$,
and $|\Lambda/\ann\TT|=|T|$ by Proposition \ref{basics}(e). 
Hence $(T,\Lambda e)$ is a support $\tau$-tilting pair for $\Lambda$.
Moreover we have $\TT=\Fac P(\TT)$ by Proposition \ref{basics}(g).

Assume conversely that $T$ is a support $\tau$-tilting $\Lambda$-module. 
Then $T$ is a $\tau$-tilting $(\Lambda/\langle e\rangle)$-module for an idempotent $e$ of $\Lambda$.
Thus $\Fac T$ is a functorially finite torsion class in $\mod(\Lambda/\langle e\rangle)$ 
such that $T\in\add P(\Fac T)$ by Proposition \ref{Auslander-Smalo}(b).
Since $|T|=|\Lambda/\langle e\rangle|$, we have $\add T=\add P(\Fac T)$ by Proposition \ref{number of summands of tautilting}. Thus $T\simeq P(\Fac T)$.
\end{proof}

We denote by $\ttilt\Lambda$\index{3tautilt@$\ttilt\Lambda$} (respectively, $\tilt\Lambda$\index{3tilt@$\tilt\Lambda$})
the set of isomorphism classes of basic $\tau$-tilting $\Lambda$-modules (respectively, tilting $\Lambda$-modules).
On the other hand, we denote by $\sftors\Lambda$\index{3sftors@$\sftors\Lambda$} (respectively, $\fftors\Lambda$\index{3fftors@$\fftors\Lambda$}) the set of sincere (respectively, faithful) functorially finite torsion classes in $\mod\Lambda$.

\begin{corollary}\label{basic bijection2}
The bijection in Theorem \ref{basic bijection} induces bijections
\begin{eqnarray*}
\ttilt\Lambda\longleftrightarrow\sftors\Lambda\ \ \ \mbox{and}\ \ \ \tilt\Lambda\longleftrightarrow\fftors\Lambda.
\end{eqnarray*}
\end{corollary}

\begin{proof}
Let $T$ be a support $\tau$-tilting $\Lambda$-module.
By Proposition \ref{sincere}, it follows that $T$ is a $\tau$-tilting $\Lambda$-module
(respectively, tilting $\Lambda$-module) if and only if $T$ is sincere
(respectively, faithful)
if and only if $\Fac T$ is sincere (respectively, faithful).
\end{proof}

We are interested in the torsion classes where our original module $U$ is a direct summand of $T=P(\TT)$, since
we would like to complete $U$ to a (support) $\tau$-tilting module. The
conditions for this to be the case are the following.

\begin{proposition}\label{prop2.1}
Let $\TT$ be a functorially finite torsion class and $U$ a $\tau$-rigid $\Lambda$-module.
Then $U\in\add P(\TT)$ if and only if $\Fac U \subseteq \TT\subseteq{^{\perp}(\tau U)}$.
\end{proposition}

\begin{proof}
We have $\TT=\Fac P(\TT)$ by Proposition \ref{basics}(g).

Assume  $\Fac U \subseteq \TT\subseteq {^{\perp}(\tau U)}$. Then $U$ is in $\TT$.
We want to show that $U$ is $\Ext$-projective in $\TT$, that is,
$\Ext_\Lambda^1(U, \TT)=0$, or equivalently
$\Hom_\Lambda(P(\TT),\tau U)=0$, by Proposition \ref{Auslander-Smalo}(a).
This follows since $P(\TT)\in\TT\subseteq{^{\perp}(\tau U)}$.
Hence $U$ is a direct summand of $P(\TT)$.

Conversely, assume $U\in\add P(\TT)$. Then we must have $U\in\TT$, and
hence $\Fac U\subseteq \TT$. 
Since $U$ is $\Ext$-projective in $\TT$, we have $\Ext_\Lambda^1(U,\TT)=0$.
Since $\TT=\Fac\TT$, we have $\Hom_\Lambda(\TT,\tau U)=0$ by Proposition \ref{Auslander-Smalo}(a).
Hence we have $\TT\subseteq{^{\perp}(\tau U)}$.
\end{proof}

We now prove the analog, for $\tau$-tilting modules, of the Bongartz
completion of classical tilting modules.

\begin{theorem}\label{thm2.3}
Let $U$ be a $\tau$-rigid $\Lambda$-module. Then
$\TT:={}^{\perp}(\tau U)$ is a sincere functorially finite torsion class and
$T:=P(\TT)$ is a $\tau$-tilting $\Lambda$-module satisfying $U\in\add T$ and $^{\perp}(\tau T)=\Fac T$.
\end{theorem}

We call $P({}^{\perp}(\tau U))$ the \emph{Bongartz completion}\index{Bongartz completion!tautilting module@$\tau$-tilting module} of $U$.

\begin{proof}
The first part follows from the following observation.

\begin{lemma}\label{lem2.4}
For any $\tau$-rigid $\Lambda$-module $U$, we have a sincere functorially finite torsion class ${^{\perp}(\tau U)}$.
\end{lemma}

\begin{proof}
When $U$ is $\tau$-rigid, then $\Sub \tau U$ is a torsionfree class
by the dual of Proposition \ref{Auslander-Smalo}(b). Then $({^{\perp}(\tau U)}, \Sub\tau U)$
is a torsion pair, and $\Sub \tau U$ and $^{\perp}(\tau U)$ are functorially finite by Proposition \ref{basics}.

Assume that $^{\perp}(\tau U)$ is not sincere. Then we have
$^{\perp}(\tau U)\subseteq \mod(\Lambda/\langle e\rangle)$ for some
primitive idempotent $e$ in $\Lambda$. The corresponding simple
$\Lambda$-module $S$ is not a composition factor of any module in
$^{\perp}(\tau U)$; in particular $\Hom({^{\perp}(\tau U)},D(e\Lambda))=0$.
Then $D(e\Lambda)$ is in $\Sub\tau U$. But this is a contradiction since $\tau U$, and hence also any module in $\Sub\tau U$, has no nonzero injective direct summands.
\end{proof}

By Corollary \ref{basic bijection2}, it follows that $T$ is a $\tau$-tilting $\Lambda$-module such that $^{\perp}(\tau U)=\Fac T$.
By Proposition \ref{prop2.1}, we have $U\in\add T$.
Clearly $^{\perp}(\tau U)\supseteq{^{\perp}(\tau T)}$
since $U$ is in $\add T$. Hence we get $\Fac T={^{\perp}(\tau U)}
\supseteq{^{\perp}(\tau T)}\supseteq\Fac T$, and consequently $^{\perp}(\tau T)=\Fac T$. 
\end{proof}

We have the following characterizations of a $\tau$-rigid module being $\tau$-tilting.

\begin{theorem}\label{3 conditions}
The following are equivalent for a $\tau$-rigid $\Lambda$-module $T$.
\begin{itemize}
\item[(a)] $T$ is $\tau$-tilting.
\item[(b)] $T$ is maximal $\tau$-rigid, i.e. if $T\oplus X$ is $\tau$-rigid for some $\Lambda$-module $X$, then $X\in\add T$.
\item[(c)] $^{\perp}(\tau T)=\Fac T$.
\item[(d)] If $\Hom_\Lambda(T,\tau X)=0$ and $\Hom_\Lambda(X,\tau T)=0$, then $X\in\add T$.
\end{itemize}
\end{theorem}

\begin{proof}
(a)$\Rightarrow$(b): Immediate from Proposition \ref{number of summands of tautilting}.


(b)$\Rightarrow$(c): Let $U$ be the Bongartz completion of $T$.
Since $T$ is maximal $\tau$-rigid, we have $T\simeq U$,
and hence $^{\perp}(\tau T)={^{\perp}(\tau U)}=\Fac U=\Fac T$, using Theorem \ref{thm2.3}.

(c)$\Rightarrow$(a): Let $T$ be $\tau$-rigid with $^{\perp}(\tau T)=\Fac T$.
Let $U$ be the Bongartz completion of $T$. Then we have
\[\Fac T={^{\perp}(\tau T)}\supseteq{^{\perp}(\tau U)}\supseteq\Fac U\supseteq\Fac T,\]
and hence all inclusions are equalities. Since $\Fac U=\Fac T$, there exists an exact sequence
\begin{equation}\label{approximation of U}
\xymatrix{0\ar[r]&Y\ar[r]& T'\ar[r]^{f}& U\ar[r]&0}
\end{equation}
where $f:T'\to U$ is a right $(\add T)$-approximation.
By the Wakamatsu-type Lemma \ref{prop1.9} we have
$\Hom_{\Lambda}(Y,\tau T)=0$, and hence $\Hom_{\Lambda}(Y,\tau U)=0$
since $^{\perp}(\tau T)={^{\perp}(\tau U)}$.
By the AR duality we have $\Ext^1_{\Lambda}(U,Y)\simeq D\overline{\Hom}_{\Lambda}(Y,\tau U)=0$, and hence the sequence \eqref{approximation of U} splits. Then it follows that $U$ is in $\add T$.
Thus $T$ is a $\tau$-tilting $\Lambda$-module.

(a)+(c)$\Rightarrow$(d): Assume that (a) and (c) hold, and $\Hom_{\Lambda}(T,\tau X)=0$ and $\Hom_{\Lambda}(X,\tau T)=0$.
Then $\Ext_{\Lambda}^{1}(X,\Fac T)=0$ by Proposition \ref{Auslander-Smalo}(a) and $X$ is in $^{\perp}\tau T=\Fac T$. 
Thus $X$ is in $\add P(\Fac T)=\add T$ by Theorem \ref{basic bijection}.

(d)$\Rightarrow$(b): This is clear.
\end{proof}

We note the following generalization.

\begin{corollary}\label{analog of 3 conditions}
The following are equivalent for a $\tau$-rigid pair $(T,P)$ for $\Lambda$.
\begin{itemize}
\item[(a)] $(T,P)$ is a support $\tau$-tilting pair for $\Lambda$.
\item[(b)] If $(T\oplus X,P)$ is $\tau$-rigid for some $\Lambda$-module $X$, then $X\in\add T$.
\item[(c)] $^\perp(\tau T)\cap P^\perp=\Fac T$.
\item[(d)] If $\Hom_\Lambda(T,\tau X)=0$, $\Hom_\Lambda(X,\tau T)=0$ and $\Hom_\Lambda(P,X)=0$, then $X\in\add T$.
\end{itemize}
\end{corollary}

\begin{proof}
In view of Lemma \ref{prop1.8}(b), the assertion follows immediately from Theorem \ref{3 conditions} by replacing $\Lambda$ by $\Lambda/\langle e\rangle$ for an idempotent $e$ of $\Lambda$ satisfying $\add P=\add\Lambda e$.
\end{proof}

In the rest of this subsection, we discuss the left-right symmetry of $\tau$-rigid modules.
It is somehow surprising that there exists a bijection between support $\tau$-tilting $\Lambda$-modules and support $\tau$-tilting $\Lambda^{\op}$-modules.
We decompose $M$ in $\mod\Lambda$ as $M=M_{\rm pr}\oplus M_{\rm np}$ where $M_{\rm pr}$ is a maximal projective direct summand of $M$.
For a $\tau$-rigid pair $(M,P)$ for $\Lambda$, let
\[(M,P)^\dagger:=(\Tr M_{\rm np}\oplus P^*,M_{\rm pr}^*)=(\Tr M\oplus P^*,M_{\rm pr}^*).\]\index{$(M,P)^\dagger$}
We denote by $\trigid\Lambda$\index{3taurigid@$\trigid\Lambda$} the set of isomorphism classes of basic $\tau$-rigid pairs of $\Lambda$.

\begin{theorem}\label{left-right symmetry}
$(-)^\dagger$ gives bijections
\begin{eqnarray*}
\trigid\Lambda\longleftrightarrow\trigid\Lambda^{\op}
\ \ \ \mbox{ and }\ \ \ \sttilt\Lambda\longleftrightarrow\sttilt\Lambda^{\op}
\end{eqnarray*}
such that $(-)^{\dagger\dagger}={\rm id}$.
\end{theorem}

For a support $\tau$-tilting $\Lambda$-module $M$, we simply write $M^\dagger:=\Tr M_{\rm np}\oplus P^*$\index{$M^{\dagger}$} where $(M,P)$ is a support $\tau$-tilting pair for $\Lambda$.

\begin{proof}
We only have to show that $(M,P)^\dagger$ is a $\tau$-rigid pair for $\Lambda^{\op}$
since the correspondence $(M,P)\mapsto(M,P)^\dagger$ is clearly an involution.
We have
\begin{equation}\label{1}
0=\Hom_\Lambda(M_{\rm np},\tau M)=\Hom_{\Lambda^{\op}}(\Tr M,DM_{\rm np})=\Hom_{\Lambda^{\op}}(\Tr M,\tau\Tr M).
\end{equation}
Moreover we have
\begin{equation}\label{2}
0=\Hom_\Lambda(M_{\rm pr},\tau M)=\Hom_{\Lambda^{\op}}(\Tr M,DM_{\rm pr})=D\Hom_{\Lambda^{\op}}(M_{\rm pr}^*,\Tr M).
\end{equation}
On the other hand, we have
\begin{equation}\label{3}
0=\Hom_\Lambda(P,M)=\Hom_\Lambda(P,M_{\rm pr})\oplus\Hom_\Lambda(P,M_{\rm np}).
\end{equation}
Thus we have
\[0=D(P^*\otimes_\Lambda M_{\rm np})=\Hom_{\Lambda^{\op}}(P^*,DM_{\rm np})=\Hom_{\Lambda^{\op}}(P^*,\tau\Tr M).\]
This together with \eqref{1} shows that $\Tr M\oplus P^*$ is a $\tau$-rigid $\Lambda^{\op}$-module.
We have $\Hom_{\Lambda^{\op}}(M_{\rm pr}^*,P^*)=0$ by \eqref{3}. This together with \eqref{2} shows that $(M,P)^\dagger$ is a $\tau$-rigid pair for $\Lambda^{\op}$.
\end{proof}

Now we discuss dual notions of $\tau$-rigid and $\tau$-tilting modules even though we do not use them in this paper.
\begin{itemize}
\item We call $M$ in $\mod\Lambda$ \emph{$\tau^-$-rigid}\index{tauzrigid module@$\tau^-$-rigid module} if $\Hom_\Lambda(\tau^-M,M)=0$.
\item We call $M$ in $\mod\Lambda$ \emph{$\tau^-$-tilting}\index{tauztilting module@$\tau^-$-tilting module} if $M$ is $\tau^-$-rigid and $|M|=|\Lambda|$.
\item We call $M$ in $\mod\Lambda$ \emph{support $\tau^-$-tilting}\index{support tauztilting module@support $\tau^-$-tilting module} if $M$ is a $\tau^-$-tilting $(\Lambda/\langle e\rangle)$-module for some idempotent $e$ of $\Lambda$.
\end{itemize}
Clearly $M$ is $\tau^-$-rigid (respectively, $\tau^-$-tilting, support $\tau^-$-tilting) $\Lambda$-module if and only if $DM$ is $\tau$-rigid (respectively, $\tau$-tilting, support $\tau$-tilting) $\Lambda^{\op}$-module.

We denote by $\cotilt\Lambda$\index{3cotilt@$\cotilt\Lambda$} (respectively, $\tcotilt\Lambda$\index{3tauztilt@$\tcotilt\Lambda$}, $\stcotilt\Lambda$\index{3stauztilt@$\stcotilt\Lambda$}) the set of isomorphism classes of basic cotilting (respectively, $\tau^-$-tilting, support $\tau^-$-tilting) $\Lambda$-modules.
On the other hand, we denote by $\ftorf\Lambda$\index{3ftorf@$\ftorf\Lambda$} the set of functorially finite torsionfree classes in $\mod\Lambda$, and by 
$\sftorf\Lambda$\index{3sftorf@$\sftorf\Lambda$} (respectively, $\fftorf\Lambda$\index{3fftorf@$\fftorf\Lambda$}) the set of sincere (respectively, faithful) functorially finite torsionfree classes in $\mod\Lambda$.
We have the following results immediately from Theorem \ref{basic bijection} and Corollary \ref{basic bijection2}.

\begin{theorem}\label{basic cobijection}
We have bijections
\begin{eqnarray*}
\stcotilt\Lambda\longleftrightarrow\ftorf\Lambda,\ \ \ 
\tcotilt\Lambda\longleftrightarrow\sftorf\Lambda\ \ \ \mbox{and}\ \ \ 
\cotilt\Lambda\longleftrightarrow\fftorf\Lambda
\end{eqnarray*}
given by $\stcotilt\Lambda\ni T\mapsto\Sub T\in\ftorf\Lambda$ and $\ftorf\Lambda\ni\FF\mapsto I(\FF)\in\stcotilt\Lambda$.
\end{theorem}

On the other hand, we have a bijection
\[\sttilt\Lambda\longleftrightarrow\stcotilt\Lambda\]
given by $(M,P)\mapsto D((M,P)^\dagger)=(\tau M\oplus\nu P,\nu M_{\rm pr})$.
Thus we have bijections
\[\ftors\Lambda\longleftrightarrow\sttilt\Lambda\longleftrightarrow\stcotilt\Lambda\longleftrightarrow\ftorf\Lambda\]
by Theorems \ref{basic bijection} and \ref{basic cobijection}.
We end this subsection with the following observation.

\begin{proposition}
\begin{itemize}
\item[(a)] The above bijections send $\TT\in\ftors\Lambda$ to $\TT^\perp\in\ftorf\Lambda$.
\item[(b)] For any support $\tau$-tilting pair $(M,P)$ for $\Lambda$, the torsion pairs $(\Fac M,M^\perp)$ and $({}^\perp(\tau M\oplus\nu P),\Sub(\tau M\oplus\nu P))$ in $\mod\Lambda$ coincide.
\end{itemize}
\end{proposition}

\begin{proof}
(b) We only have to show $\Fac M={}^\perp(\tau M\oplus\nu P)$.
It follows from Proposition \ref{Auslander-Smalo}(b) and its dual that
$(\Fac M,M^\perp)$ and $({}^\perp(\tau M\oplus\nu P),\Sub(\tau M\oplus\nu P))$
are torsion pairs in $\mod\Lambda$.
They coincide since $\Fac M={}^\perp(\tau M)\cap P^\perp={}^\perp(\tau M\oplus\nu P)$ holds by Corollary \ref{analog of 3 conditions}(c).


(a) Let $\TT\in\ftors\Lambda$ and $(M,P)$ be the corresponding support $\tau$-tilting pair
for $\Lambda$. Since $\TT^\perp=M^\perp$ and $D(M^\dagger)=\tau M\oplus\nu P$,
the assertion follows from (b).
\end{proof}

\subsection{Mutation of support $\tau$-tilting modules}\label{ssec2.5}

In this section we prove our main result on complements for almost complete support $\tau$-tilting pairs.
Let us start with the following result.

\begin{proposition}\label{prop2.7}
Let $T$ be a basic $\tau$-rigid module which is not $\tau$-tilting. Then there
are at least two basic support $\tau$-tilting modules which have $T$ as a direct summand.
\end{proposition}

\begin{proof}
By Theorem \ref{3 conditions}, $\TT_1=\Fac T$ is properly contained in $\TT_2={^{\perp}(\tau T)}$. 
By Theorem \ref{basic bijection} and Lemma \ref{lem2.4}, we have two different support $\tau$-tilting modules $P(\TT_1)$ and $P(\TT_2)$ up to isomorphism.
By Proposition \ref{prop2.1}, they are extensions of $T$.
\end{proof}

Our aim is to prove the following result.

\begin{theorem}\label{2 complements}
Let $\Lambda$ be a finite dimensional $k$-algebra.
Then any basic almost complete support $\tau$-tilting pair $(U,Q)$ for $\Lambda$ is a direct summand of exactly two basic support $\tau$-tilting pairs $(T,P)$ and $(T',P')$ for $\Lambda$.
Moreover we have $\{\Fac T,\Fac T'\}=\{\Fac U,{}^\perp(\tau U)\cap Q^\perp\}$.
\end{theorem}

Before proving Theorem \ref{2 complements}, we introduce a notion of mutation.

\begin{definition}
Two basic support $\tau$-tilting pairs $(T,P)$ and $(T',P')$ for $\Lambda$ are
said to be \emph{mutations}\index{mutation!support tautilting pair@support $\tau$-tilting pair} of each other 
if there exists a basic almost complete support $\tau$-tilting pair $(U,Q)$ which is a direct summand of $(T,P)$ and $(T',P')$.
In this case we write $(T',P')=\mu_X(T,P)$ or simply $T'=\mu_X(T)$\index{3mu@$\mu$} if
$X$ is an indecomposable $\Lambda$-module satisfying either $T=U\oplus X$ or $P=Q\oplus X$.
\end{definition}

We can also describe mutation as follows:
Let $(T,P)$ be a basic support $\tau$-tilting pair for $\Lambda$, and $X$ an indecomposable direct summand of either $T$ or $P$.
\begin{itemize}
\item[(a)] If $X$ is a direct summand of $T$, precisely one of the following holds.
\begin{itemize}
\item[$\bullet$] There exists an indecomposable $\Lambda$-module $Y$ such that $X\ {\not\simeq}\ Y$ and $\mu_X(T,P):=(T/X\oplus Y,P)$ is a basic support $\tau$-tilting pair for $\Lambda$.
\item[$\bullet$] There exists an indecomposable projective $\Lambda$-module $Y$ such that $\mu_X(T,P):=(T/X,P\oplus Y)$ is a basic support $\tau$-tilting pair for $\Lambda$.
\end{itemize}
\item[(b)] If $X$ is a direct summand of $P$, there exists an indecomposable $\Lambda$-module $Y$ such that $\mu_X(T,P):=(T\oplus Y,P/X)$ is a basic support $\tau$-tilting pair for $\Lambda$.
\end{itemize}
Moreover, such a module $Y$ in each case is unique up to isomorphism.

In the rest of this subsection, we give a proof of Theorem \ref{2 complements}. The following is the first step.

\begin{lemma}\label{approximation lemma}
Let $(T,P)$ be a $\tau$-rigid pair for $\Lambda$. If $U$ is a $\tau$-rigid $\Lambda$-module satisfying $^{\perp}(\tau T)\cap P^\perp\subseteq{{^\perp}(\tau U)}$, then there is an exact sequence
$U\xto{f}T'\to C\to0$ satisfying the following conditions.
\begin{itemize}
\item $f$ is a minimal left $(\Fac T)$-approximation.
\item $T'$ is in $\add T$, $C$ is in $\add P(\Fac T)$ and $\add T'\cap\add C=0$.
\end{itemize}
\end{lemma}

\begin{proof}
Consider the exact sequence $U\xto{f}T'\xto{g}C\to0$, where $f$ is a minimal left $(\add T)$-approximation. Then $g\in\rad(T',C)$.

(i) $f$ is a minimal left $(\Fac T)$-approximation:
Take any $X\in\Fac T$ and $s:U\to X$. By the Wakamatsu-type Lemma \ref{prop1.9},
there exists an exact sequence
$$0\to Y\to T''\xrightarrow{h}X\to0$$
where $h$ is a right $(\add T)$-approximation and $Y\in{^{\perp}}(\tau T)$.
Moreover we have $Y\in P^\perp$ since $T''\in P^\perp$. By the assumption that
$^{\perp}(\tau T)\cap P^\perp\subseteq{^{\perp}(\tau U)}$, we have $\Hom_\Lambda(Y,\tau U)=0$,
hence $\Ext_\Lambda^1(U,Y)=0$. Then we have an exact sequence
$$\Hom_\Lambda(U,T'')\to\Hom_\Lambda(U,X)\to\Ext_\Lambda^1(U,Y)=0.$$
Thus there is some $t:U\to T''$ such that $s=ht$.
$$\xymatrix{
&&&U\ar@{.>}[dl]_t\ar[d]^s\ar[r]^f&T'\\
0\ar[r]&Y\ar[r]&T''\ar[r]^h&X\ar[r] &0
}$$
Since $T''\in \add T$ and $f$ is a left $(\add T)$-approximation, there is some $u:T'\to T''$ such that $t=uf$.
Hence we have $hu:T'\to X$ such that $(hu)f=ht=s$, and the claim follows. 

(ii) $C\in\add P(\Fac T)$:
We have an exact sequence $0\to\im f\xto{i} T'\to C\to0$, which gives rise to an exact sequence
$$\Hom_\Lambda(T',\Fac T)\xto{(i,\Fac T)}\Hom_\Lambda(\im f,\Fac T)\to\Ext_\Lambda^1(C,\Fac T)\to\Ext_\Lambda^1(T',\Fac T).$$
We know from (i) that $(f,\Fac T):\Hom_\Lambda(T',\Fac T)\to\Hom_\Lambda(U,\Fac T)$ is surjective, and hence $(i,\Fac T)$ is surjective. 
Further, $\Ext_\Lambda^1(T',\Fac T)=0$ by Proposition \ref{Auslander-Smalo} since $T'$ is in $\add T$ and $T$ is $\tau$-rigid.
Then it follows that $\Ext_\Lambda^1(C,\Fac T)=0$. Since $C\in\Fac T$, this means that $C$ is $\Ext$-projective in $\Fac T$.

(iii) $\add T'\cap \add C=0$: To show this, it is clearly
sufficient to show $\Hom_\Lambda(T',C)\subseteq\rad(T',C)$.

Let $s: T'\to C$ be an arbitrary map. We have an exact sequence
$\Hom_\Lambda(U,T')\to\Hom_\Lambda(U,C)\to\Ext_\Lambda^1(U,\im f)$.
Since $\Ext_\Lambda^1(U,\im f)=0$
because $\im f$ is in $\Fac U$, and $U$ is $\tau$-tilting, there is a map
$t: U\to T'$ such that $sf=gt$. Since $f$ is a left $(\add T)$-approximation,
and $T'$ is in $\add T$, there is a map $u:T'\to T'$ such that $t=uf$. Then $(s-gu)f=sf-gt=0$, hence there is some $v: C\to C$ such that $s-gu=vg$, and hence $s=gu+vg$.
$$\xymatrix@C0.6cm@R0.3cm{
&& U\ar[rr]^f\ar[dd]^t&& T'\ar[rr]^g\ar[dd]^s\ar@{.>}[ddll]_u &&
C\ar[rr]\ar@{.>}[ddll]_v && 0\\
\\
U\ar@{->>}[dr]\ar[rr]^f&& T'\ar[rr]^g && C\ar[rr]&&0\\
&\im f\ar@{>->}[ur]&&&&
}$$
Since $g\in \rad(T',C)$, it follows that $s\in\rad(T',C)$. Hence
$\Hom_\Lambda(T',C)\subseteq \rad(T',C)$, and consequently $\add T'\cap \add C=0$. 
\end{proof}

The following information on the previous lemma is useful.

\begin{lemma}\label{approximation is surjective}
In Lemma \ref{approximation lemma}, assume $C=0$.
Then $f:U\to T'$ induces an isomorphism $U/\langle e\rangle U\simeq T'$ for a maximal idempotent $e$ of $\Lambda$ satisfying $eT=0$.
In particular, if $T$ is sincere, then $U\simeq T'$.
\end{lemma}

\begin{proof}
By our assumption, we have an exact sequence
\begin{equation}\label{canonical}
\xymatrix{0\ar[r]&\Ker f\ar[r]&U\ar[r]^f&T'\ar[r]&0.}
\end{equation}
Applying $\Hom_\Lambda(-,\Fac T)$, we have an exact sequence
\[\Hom_\Lambda(T',\Fac T)\xto{(f,\Fac T)}\Hom_\Lambda(U,\Fac T)\to\Hom_\Lambda(\Ker f,\Fac T)\to\Ext^1_\Lambda(T',\Fac T).\]
We have $\Ext_\Lambda^1(T',\Fac T)=0$ because $T'$ is in $\add T$ and $T$ is $\tau$-tilting.
Since $(f,\Fac T)$ is surjective, it follows that $\Hom_\Lambda(\Ker f,\Fac T)=0$ and so $\Ker f\in{}^{\perp}(\Fac T)$.
On the other hand, since $T$ is a sincere $(\Lambda/\langle e\rangle)$-module, $\mod(\Lambda/\langle e\rangle)$ is the smallest torsionfree class of $\mod\Lambda$ containing $\Fac T$.
Thus we have a torsion pair $({}^{\perp}(\Fac T),\mod(\Lambda/\langle e\rangle))$, 
and the canonical sequence for $X$ associated with this torsion pair is given by
\[\xymatrix{0\ar[r]&\langle e\rangle X\ar[r]&X\ar[r]&X/\langle e\rangle X\ar[r]&0.}\]
Since $\Ker f\in{}^{\perp}(\Fac T)$ and $T'\in\Fac T\subseteq\mod(\Lambda/\langle e\rangle)$, the canonical sequence of $U$ is given by \eqref{canonical}.
Thus we have $U/\langle e\rangle U\simeq T'$.
\end{proof}

In the next result we prove a useful restriction on $X$ when $T=X\oplus U$ is $\tau$-tilting and $X $ is indecomposable.

\begin{proposition}\label{prop2.10}
Let $T=X\oplus U$ be a basic $\tau$-tilting $\Lambda$-module, with $X$ indecomposable.
Then exactly one of $^{\perp}(\tau U)\subseteq{^{\perp}(\tau X)}$ and $X\in\Fac U$ holds.
\end{proposition}

\begin{proof}
First we assume that $^{\perp}(\tau U)\subseteq{^{\perp}(\tau X)}$ and $X\in\Fac U$ both hold. Then we have
\[\Fac U=\Fac T={}^\perp(\tau T)={}^\perp(\tau U),\]
which implies that $U$ is $\tau$-tilting by Theorem \ref{3 conditions}, a contradiction.

Let $Y\oplus U$ be the Bongartz completion of $U$.
Then we have $^{\perp}\tau(Y\oplus U)={}^{\perp}(\tau U)\supseteq{^{\perp}\tau T}$.
Using the triple $(T,0,Y\oplus U)$ instead of $(T,P,U)$ in Lemma \ref{approximation lemma},
there is an exact sequence
\[\xymatrix{Y\oplus U\ar[r]^{{f\ 0\choose0\ 1}}&T'\oplus U\ar[r]&T''\ar[r]&0,}\]
where $f:Y\to T'$ and ${f\ 0\choose0\ 1}:Y\oplus U\to T'\oplus U$ are minimal left $(\Fac T)$-approximations, $T'$ and $T''$ are in $\add T$ and $\add(T'\oplus U)\cap \add T''=0$.
Then we have $T''\in\add X$. 

Assume first $T''\neq0$. Then $T''\simeq X^\ell$ for some $\ell\ge1$, so we have $T'\in\add U$.
Since we have a surjective map $T'\to T''$, we have $X\in\Fac T'\subseteq\Fac U$.

Assume now that $T''=0$. 
Applying Lemma \ref{approximation is surjective}, we have that ${{f\ 0\choose0\ 1}}:Y\oplus U\to T'\oplus U$ is an isomorphism since $T$ is sincere.
Thus $Y\in\add T$, and we must have $Y\simeq X$.
Thus $^\perp(\tau X)={}^\perp(\tau Y)\supseteq{}^\perp(\tau U)$.
\end{proof}

Now we are ready to prove Theorem \ref{2 complements}.

(i) First we assume that $Q=0$ (i.e. $U$ is an almost complete $\tau$-tilting module).

In view of Proposition \ref{prop2.7} it only remains to show that
there are at most two extensions of $U$ to a support $\tau$-tilting module.
Using the bijection in Theorem \ref{basic bijection}, we only have to show that for any support $\tau$-tilting module $X\oplus U$, the torsion class $\Fac(X\oplus U)$ is either $\Fac U$ or $^{\perp}(\tau U)$.
If $X=0$ (i.e. $U$ is a support $\tau$-tilting module), then this is clear.
If $X\neq0$, then $X\oplus U$ is a $\tau$-tilting $\Lambda$-module.
Moreover by Proposition \ref{prop2.10} either $X\in\Fac U$ or $^{\perp}(\tau U)\subseteq{^{\perp}(\tau X)}$ holds.
If $X\in\Fac U$, then we have $\Fac(X\oplus U)=\Fac U$. If $^{\perp}(\tau U)\subseteq{^{\perp}(\tau X)}$, then we have $\Fac(X\oplus U)={^{\perp}(\tau(X\oplus U))}={^{\perp}(\tau U)}$.
Thus the assertion follows.

(ii) Let $(U,Q)$ be a basic almost complete support $\tau$-tilting pair for $\Lambda$
and $e$ be an idempotent of $\Lambda$ such that $\add Q=\add\Lambda e$.
Then $U$ is an almost complete $\tau$-tilting $(\Lambda/\langle e\rangle)$-module
by Proposition \ref{two notion of tau-rigidity}(a).
It follows from (i) that $U$ is a direct summand of exactly two basic support
$\tau$-tilting $(\Lambda/\langle e\rangle)$-modules.
Thus the assertion follows since basic support $\tau$-tilting $(\Lambda/\langle e\rangle)$-modules 
which have $U$ as a direct summand correspond bijectively to
basic support $\tau$-tilting pairs for $\Lambda$ which have $(U,Q)$ as a direct summand.
\qed

\medskip
The following special case of Lemma \ref{approximation lemma} is useful.

\begin{proposition}\label{two term left resolution}
Let $T$ be a support $\tau$-tilting $\Lambda$-module.
Assume that one of the following conditions is satisfied.
\begin{itemize}
\item[(i)] $U$ is a $\tau$-rigid $\Lambda$-module such that $\Fac T\subseteq{}^\perp(\tau U)$.
\item[(ii)] $U$ is a support $\tau$-tilting $\Lambda$-module such that $U\ge T$.
\end{itemize}
Then there exists an exact sequence $U\xto{f}T^0\to T^1\to0$ such that $f$ is a minimal
left $(\Fac T)$-approximation of $U$ and $T^0$ and $T^1$ are in $\add T$ and satisfy $\add T^0\cap\add T^1=0$.
\end{proposition}

\begin{proof}
Let $(T,P)$ be a support $\tau$-tilting pair for $\Lambda$. Then ${}^\perp(\tau T)\cap P^\perp=\Fac T$ holds by Corollary \ref{analog of 3 conditions}(c).
Thus ${}^\perp(\tau T)\cap P^\perp\subseteq{}^\perp(\tau U)$ holds for both cases.
Hence the assertion is immediate from Lemma \ref{approximation lemma} since $C$ is in
$\add P(\Fac T)=\add T$ by Theorem \ref{basic bijection}.
\end{proof}

The following well-known result \cite{HU1} can be shown as an application of our results.

\begin{corollary}
Let $\Lambda$ be a finite dimensional $k$-algebra and $U$ a basic almost complete tilting $\Lambda$-module.
Then $U$ is faithful if and only if $U$ is a direct summand of precisely two basic tilting $\Lambda$-modules.
\end{corollary}

\begin{proof}
It follows from Theorem \ref{2 complements} that $U$ is a direct summand of exactly two basic support $\tau$-tilting $\Lambda$-modules $T$ and $T'$ such that $\Fac T=\Fac U$.
If $U$ is faithful, then $T$ and $T'$ are tilting $\Lambda$-modules by Proposition \ref{sincere}(b). Thus the `only if' part follows.
If $U$ is not faithful, then $T$ is not a tilting $\Lambda$-module since it is not faithful because $\Fac T=\Fac U$. Thus the `if' part follows.
\end{proof}

\subsection{Partial order, exchange sequences and Hasse quiver}

In this section we investigate two quivers. One is defined by partial order, and the other one by mutation.
We show that they coincide.

Since we have a bijection $T\mapsto\Fac T$ between $\sttilt\Lambda$ and $\ftors\Lambda$,
then inclusion in $\ftors\Lambda$ gives rise to a partial order on $\sttilt\Lambda$,
and we have an associated Hasse quiver. Note that $\sttilt\Lambda$ has a unique maximal element $\Lambda$ and a unique minimal element $0$.

The following description of when $T\ge U$ holds will be useful.

\begin{lemma}\label{partial order of tautilting}
Let $(T,P)$ and $(U,Q)$ be support $\tau$-tilting pairs for $\Lambda$.
Then the following conditions are equivalent.
\begin{itemize}
\item[(a)] $T\ge U$.
\item[(b)] $\Hom_\Lambda(U,\tau T)=0$ and $\add P\subseteq\add Q$.
\item[(c)] $\Hom_\Lambda(U_{\rm np},\tau T_{\rm np})=0$, $\add T_{\rm pr}\supseteq\add U_{\rm pr}$ and $\add P\subseteq\add Q$.
\end{itemize}
\end{lemma}

\begin{proof}
(a)$\Rightarrow$(c) Since $\Fac T\supseteq\Fac U$, we have $\add T_{\rm pr}\supseteq\add U_{\rm pr}$ and $\Hom_\Lambda(U,\tau T)=0$. Moreover $\add P\subseteq\add Q$ holds by Proposition \ref{sincere}(a).

(b)$\Rightarrow$(a) We have $\Fac T={}^\perp(\tau T)\cap P^\perp$ by Corollary \ref{analog of 3 conditions}(c).
Since $\add P\subseteq\add Q$, we have $U\in Q^\perp\subseteq P^\perp$.
Since $\Hom_\Lambda(U,\tau T)=0$, we have $U\in{}^\perp(\tau T)\cap P^\perp=\Fac T$, which implies $\Fac T\supseteq\Fac U$.

(c)$\Rightarrow$(b) This is clear.
\end{proof}

Also we shall need the following.

\begin{proposition}\label{common summand}
Let $T,U,V\in\sttilt\Lambda$ such that $T\ge U\ge V$. Then $\add T\cap\add V\subseteq\add U$.
\end{proposition}

\begin{proof}
Clearly we have $P(\Fac T)\cap\Fac U\subseteq P(\Fac U)=\add U$.
Thus we have $\add T\cap\add V\subseteq P(\Fac T)\cap\Fac U\subseteq\add U$.
\end{proof}

The following observation is immediate.

\begin{proposition}\label{reverse}
\begin{itemize}
\item[(a)] For any idempotent $e$ of $\Lambda$, the inclusion $\sttilt(\Lambda/\langle e\rangle)\to\sttilt\Lambda$ preserves the partial order.
\item[(b)] The bijection $(-)^\dagger:\sttilt\Lambda\to\sttilt\Lambda^{\op}$ in Theorem \ref{left-right symmetry} reverses the partial order.
\end{itemize}
\end{proposition}

\begin{proof}
(a) This is clear.

(b) Let $(T,P)$ and $(U,Q)$ be support $\tau$-tilting pairs of $\Lambda$.
By Lemma \ref{partial order of tautilting}, $T\ge U$ if and only if $\Hom_\Lambda(U_{\rm np},\tau T_{\rm np})=0$, $\add T_{\rm pr}\supseteq\add U_{\rm pr}$ and $\add P\subseteq\add Q$.
This is equivalent to $\Hom_{\Lambda^{\op}}(\Tr T_{\rm np},\tau\Tr U_{\rm np})=0$, $\add T_{\rm pr}^*\supseteq\add U_{\rm pr}^*$ and $\add P^*\subseteq\add Q^*$.
By Lemma \ref{partial order of tautilting} again, this is equivalent to $(\Tr T_{\rm np}\oplus P^*,T_{\rm pr}^*)\le(\Tr U_{\rm np}\oplus Q^*,U_{\rm pr}^*)$.
\end{proof}

In the rest of this section, we study a relationship between partial order and mutation.

\begin{definition-proposition}
Let $T=X\oplus U$ and $T'$ be support $\tau$-tilting $\Lambda$-modules such that
$T'=\mu_X(T)$ for some indecomposable $\Lambda$-module $X$.
Then either $T>T'$ or $T<T'$ holds by Theorem \ref{2 complements}.
We say that $T'$ is a \emph{left mutation}\index{left mutation!support tautilting module@support $\tau$-tilting module} 
(respectively, \emph{right mutation}\index{right mutation!support tautilting module@support $\tau$-tilting module}) of $T$ and 
we write $T'=\mu^{-}_X(T)$\index{3mu@$\mu^{+}$, $\mu^{-}$} (respectively, $T'=\mu^{+}_X(T)$\index{3mu@$\mu^{+}$, $\mu^{-}$}) if the following
equivalent conditions are satisfied.
\begin{itemize}
\item[(a)] $T>T'$ (respectively, $T<T'$).
\item[(b)] $X\notin\Fac U$ (respectively, $X\in\Fac U$).
\item[(c)] $^\perp(\tau X)\supseteq{}^\perp(\tau U)$ (respectively, $^\perp(\tau X)\ {\not \supseteq}\ {}^\perp(\tau U)$).
\end{itemize}
If $T$ is a $\tau$-tilting $\Lambda$-module, then the following condition is also equivalent to the above conditions.
\begin{itemize}
\item[(d)] $T$ is a Bongartz completion of $U$ (respectively, $T$ is a non-Bongartz completion of $U$).
\end{itemize}
\end{definition-proposition}

\begin{proof}
This follows immediately from Theorem \ref{2 complements} and Proposition \ref{prop2.10}.
\end{proof}

\begin{definition}\label{support tau-tilting quiver}
We define the \emph{support $\tau$-tilting quiver}\index{support tautilting quiver@support $\tau$-tilting quiver} 
$\Q(\sttilt\Lambda)$\index{3qstautilt@$\Q(\sttilt\Lambda)$} of $\Lambda$ as follows:
\begin{itemize}
\item The set of vertices is $\sttilt\Lambda$.
\item We draw an arrow from $T$ to $U$ if $U$ is a left mutation of $T$.
\end{itemize}
\end{definition}

Next we show that one can calculate left mutation of support $\tau$-tilting
$\Lambda$-modules by exchange sequences which are constructed from left approximations.

\begin{theorem}\label{thm2.8}
Let $T=X\oplus U$ be a basic $\tau$-tilting module which is the Bongartz completion of $U$, where $X$ is indecomposable. 
Let $X\xto{f}U'\xto{g} Y\to 0$ be an exact sequence, where $f$ is a minimal left $(\add U)$-approximation.
Then we have the following.
\begin{itemize}
\item[(a)] If $U$ is not sincere, then $Y=0$.  In this case $U=\mu^{-}_X(T)$ holds and this is a basic support $\tau$-tilting $\Lambda$-module which is not $\tau$-tilting.
\item[(b)] If $U$ is sincere, then $Y$ is a direct sum of copies of an indecomposable $\Lambda$-module $Y_1$ and is not in $\add T$.
In this case $Y_1\oplus U=\mu^{-}_X(T)$ holds and this is a basic $\tau$-tilting $\Lambda$-module.
\end{itemize}
\end{theorem}

\begin{proof}
We first make some preliminary observations.
We have $^{\perp}(\tau U)\subseteq{^{\perp}(\tau X)}$ because $T$ is a Bongartz completion of $U$.
By Lemma \ref{approximation lemma}, we have an exact sequence
\[X\xto{f}U'\xto{g}Y\to0\]
such that $U'$ is in $\add U$, $Y$ is in $\add P(\Fac U)$, $\add U'\cap\add Y=0$ and $f$ is a left $(\Fac U)$-approximation.
We have $\Ext_\Lambda^1(Y,\Fac U)=0$ since $Y\in\add P(\Fac U)$, and hence
$\Hom_\Lambda(U,\tau Y)=0$ by Proposition \ref{Auslander-Smalo}.
We have an injective map $\Hom_{\Lambda}(Y,\tau(Y\oplus U))\to\Hom_{\Lambda}(U',\tau(Y\oplus U))$.
Since $U$ is $\tau$-rigid, we have that $\Hom_{\Lambda}(U',\tau(Y\oplus U))=0$, and consequently $\Hom_{\Lambda}(Y,\tau(Y\oplus U))=0$.
It follows that $Y\oplus U$ is $\tau$-rigid.

We show that $g:U'\to Y$ is a right $(\add T)$-approximation.
To see this, consider the exact sequence
$$\Hom_\Lambda(T,U')\to\Hom_\Lambda(T,Y)\to\Ext_\Lambda^1(T,\im f).$$
Since $\im f\in\Fac T$, we have $\Ext_\Lambda^1(T,\im f)=0$, which proves the claim.

We have that $Y$ does not have any indecomposable direct summand from $\add T$.
For if $T'$ in $\add T$ is an indecomposable direct summand of $Y$, then the natural inclusion $T'\to Y$ factors through $g:U'\to Y$.
This contradicts the fact that $f:X\to U'$ is left minimal.

Now we are ready to prove the claims (a) and (b).

(a) Assume first that $U$ is not sincere. Let $e$ be a primitive
idempotent with $eU=0$. Then $U$ is a $\tau$-rigid $(\Lambda/\langle e\rangle)$-module.
Since $|U|=|\Lambda|-1=|\Lambda/\langle e\rangle|$, 
we have that $U$ is a $\tau$-tilting $(\Lambda/\langle e\rangle)$-module,
and hence a support $\tau$-tilting $\Lambda$-module which is not $\tau$-tilting.

(b) Next assume that $U$ is sincere.
Since we have already shown that $Y\oplus U$ is $\tau$-rigid and $Y\notin\add T$, it is enough to show $Y\neq0$.
Otherwise we have $X\simeq U'$ by Lemma \ref{approximation is surjective} since $U$ is sincere.
This is not possible since $U'$ is in $\add U$, but $X$ is not. Hence it follows that $Y\neq0$.
\end{proof}

We do not know the answer to the following.

\begin{question}
Is $Y$ always indecomposable in Theorem \ref{thm2.8}(b)?
\end{question}

Note that right mutation can not be calculated as directly as left mutation.

\begin{remark}
Let $T$ and $T'$ be support $\tau$-tilting $\Lambda$-modules such that
$T'=\mu_X(T)$ for some indecomposable $\Lambda$-module $X$.
\begin{itemize}
\item[(a)] If $T'=\mu^{-}_X(T)$, then we can calculate $T'$ by applying Theorem \ref{thm2.8}.
\item[(b)] If $T'=\mu^{+}_X(T)$, then we can calculate $T'$ using the following three steps:
First calculate $T^\dagger$. Then calculate $T'{}^\dagger$ by applying
Theorem \ref{thm2.8} to $T^\dagger$. Finally calculate $T'$ by applying
$(-)^\dagger$ to $T'{}^\dagger$.
\end{itemize}
\end{remark}

Our next main result is the following.

\begin{theorem}\label{Hasse is mutation}
For $T,U\in\sttilt\Lambda$, the following conditions are equivalent.
\begin{itemize}
\item[(a)] $U$ is a left mutation of $T$.
\item[(b)] $T$ is a right mutation of $U$.
\item[(c)] $T>U$ and there is no $V\in\sttilt\Lambda$ such that $T>V>U$.\end{itemize}
\end{theorem}

Before proving Theorem \ref{Hasse is mutation}, we give the following result as a direct consequence.

\begin{corollary}\label{Hasse is mutation 2}
The support $\tau$-tilting quiver $\Q(\sttilt\Lambda)$ is the Hasse quiver of the partially ordered set $\sttilt\Lambda$.
\end{corollary}

The following analog of \cite[Proposition 2.36]{AI} is a main step to prove Theorem \ref{Hasse is mutation}.

\begin{theorem}\label{make closer}
Let $U$ and $T$ be basic support $\tau$-tilting $\Lambda$-modules such that $U>T$. Then:
\begin{itemize}
\item[(a)] There exists a right mutation $V$ of $T$ such that $U\ge V$.
\item[(b)] There exists a left mutation $V'$ of $U$ such that $V'\ge T$.
\end{itemize}
\end{theorem}

Before proving Theorem \ref{make closer}, we finish the proof of Theorem
\ref{Hasse is mutation} by using Theorem \ref{make closer}.

(a)$\Leftrightarrow$(b) Immediate from the definitions.

(a)$\Rightarrow$(c) Assume that $V\in\sttilt\Lambda$ satisfies $T>V\ge U$.
Then we have $\add T\cap\add U\subseteq\add V$ by Proposition \ref{common summand}.
Thus $T$ and $V$ have an almost complete support $\tau$-tilting pair for $\Lambda$ as a common direct summand.
Hence we have $V\simeq U$ by Theorem \ref{2 complements}.

(c)$\Rightarrow$(a) By Theorem \ref{make closer}, there exists a left mutation $V$ of $T$ such that $T>V\ge U$. 
Then $V\simeq U$ by our assumption. Thus $U$ is a left mutation of $T$.
\qed

\medskip
To prove Theorem \ref{make closer}, we shall need the following results.

\begin{lemma}\label{make closer lemma}
Let $U$ and $T$ be basic support $\tau$-tilting $\Lambda$-modules such that $U>T$.
Let $U\xto{f}T^0\to T^1\to0$ be an exact sequence as given in Proposition \ref{two term left resolution}.
If $X$ is an indecomposable direct summand of $T$ which does not belong to $\add T^0$,
then we have $U\ge\mu_X(T)>T$.
\end{lemma}

\begin{proof}
First we show $\mu_X(T)>T$.
Since $X$ is in $\Fac T\subseteq \Fac U$, there exists a surjective map $a:U^\ell\to X$
for some $\ell>0$. Since $f^\ell:U^\ell\to (T^0)^\ell$ is a left $(\add T)$-approximation, $a$ factors through $f^\ell$ and we have $X\in\Fac T^0$.
It follows from $X\notin\add T^0$ that $X\in\Fac T^0\subseteq\Fac\mu_X(T)$.
Thus $\Fac T\subseteq\Fac\mu_X(T)$ and we have $\mu_X(T)>T$.

Next we show $U\ge\mu_X(T)$.
Let $(U,\Lambda e)$ and $(T,\Lambda e')$ be support $\tau$-tilting pairs for $\Lambda$.
By Proposition \ref{reverse}(b), we know that
$U^\dagger=\Tr U\oplus e\Lambda$ and $T^\dagger=\Tr T\oplus e'\Lambda$ are support
$\tau$-tilting $\Lambda^{\op}$-modules such that $U^\dagger<T^\dagger$.
In particular, any minimal right $(\add T^\dagger)$-approximation
\begin{equation}\label{Tr T oplus e'Lambda approximation}
\Tr T_0\oplus P\to U^\dagger
\end{equation}
of $U^\dagger$ with $T_0\in\add T_{\rm np}$ and $P\in\add e'\Lambda$ is surjective.
The following observation shows $T_0\in\add T^0$.

\begin{lemma}
Let $X$ and $Y$ be in $\mod\Lambda$ and $P$ in $\proj\Lambda^{\op}$.
Let $f:Y\to X^0$ be a left $(\add X)$-approximation of $Y$ and
$g:\Tr X_0\oplus P_0\to\Tr Y$ be a minimal right $(\add\Tr X\oplus P)$-approximation
of $\Tr Y$ with $X_0\in\add X_{\rm np}$ and $P_0\in\add P$.
If $g$ is surjective, then $X_0$ is a direct summand of $X^0$.
\end{lemma}

\begin{proof}
Assume that $g$ is surjective and consider the exact sequence 
\[\xymatrix{
0\ar[r]&K\ar[r]^(.3)h&\Tr X_0\oplus P_0\ar[r]^(.6)g&\Tr Y\ar[r]&0.}\]
Then $h$ is in $\rad(K,\Tr X_0\oplus P_0)$ since $g$ is right minimal.
It is easy to see that in the stable category $\underline{\mod}\Lambda^{\op}$, 
a pseudokernel of $g$ is given by $h$, which is in the radical of $\underline{\mod}\Lambda^{\op}$.
In particular, $g$ is a minimal right $(\add\Tr X)$-approximation in
 $\underline{\mod}\Lambda^{\op}$.
Since $\Tr:\underline{\mod}\Lambda\to\underline{\mod}\Lambda^{\op}$ is a duality,
we have that $\Tr g:\Tr\Tr Y\to\Tr(\Tr X_0\oplus P_0)=X_0$ is a minimal left
$(\add X)$-approximation of $\Tr\Tr Y$ in $\underline{\mod}\Lambda$.
On the other hand, $f:Y\to X^0$ is clearly a left $(\add X)$-approximation of $Y$ in $\underline{\mod}\Lambda$.
Since $\Tr\Tr Y$ is a direct summand of $Y$, we have that
$X_0$ is a direct summand of $X^0$ in $\underline{\mod}\Lambda$.
Thus the assertion follows.
\end{proof}

We now finish the proof of Lemma \ref{make closer lemma}.

Since $T_0\in\add T^0$ and $X\notin\add T^0$, we have $X\notin\add T_0$ and hence
$U^\dagger\in\Fac(\Tr(T/X)\oplus e'\Lambda)$ by \eqref{Tr T oplus e'Lambda approximation}.
Hence we have $U^\dagger\le\mu_X(T)^\dagger$, which implies $U\ge\mu_X(T)$ by Proposition \ref{reverse}(b).
\end{proof}

Now we are ready to prove Theorem \ref{make closer}.

We only prove (a) since (b) follows from (a) and Proposition \ref{reverse}(b).

(i) Let $(U,\Lambda e)$ and $(T,\Lambda e')$ be support $\tau$-tilting pairs
for $\Lambda$. Let
\begin{equation}\label{T approximation of U}
\xymatrix{
U\ar[r]&T^0\ar[r]&T^1\ar[r]&0
}\end{equation}
be an exact sequence given by Proposition \ref{two term left resolution}.
If $T\notin\add T^0$, then any indecomposable direct summand $X$ of $T$
which is not in $\add T^0$ satisfies $U\ge\mu_X(T)>T$ by Lemma \ref{make closer lemma}.
Thus we assume $T\in\add T^0$ in the rest of proof.
Since $\add T^0\cap\add T^1=0$, we have $T^1=0$ which implies
$T^0=U/\langle e'\rangle U$ by Lemma \ref{approximation is surjective}.

(ii) By Proposition \ref{reverse}(b), we know that
$U^\dagger=\Tr U\oplus e\Lambda$ and $T^\dagger=\Tr T\oplus e'\Lambda$ are support
$\tau$-tilting $\Lambda^{\op}$-modules such that $U^\dagger<T^\dagger$. Let
\[\xymatrix{T^\dagger_0 \ar[r]^(0.6){f}&U^\dagger\ar[r]&0}\]
be a minimal right $(\add T^\dagger)$-approximation of $U^\dagger$.
If $e'\Lambda\notin\add T^\dagger_0$, then any indecomposable direct summand
$Q$ of $e'\Lambda$ which is not in $\add T^\dagger_0$ satisfies
$U^\dagger\in\Fac(T^\dagger/Q)$.
Thus we have $U^\dagger\le\mu_Q(T^\dagger)$
and $U\ge\mu_{Q^*}(T)>T$ by Proposition \ref{reverse}.
We assume $e'\Lambda\in\add T^\dagger_0$ in the rest of proof.

(iii) We show that there exists an exact sequence
\begin{equation}\label{another sequence for Tr U}
\xymatrix{P_1\ar[r]^(0.3)a&\Tr T^0\oplus P_0\ar[r]&\Tr U\ar[r]&0}
\end{equation}
in $\mod\Lambda^{\op}$ such that $P_0\in\proj\Lambda^{\op}$, $P_1\in\add e'\Lambda$,
$a\in\rad(P_1,\Tr T^0\oplus P_0)$ and the map
\begin{equation}\label{to Tr U}
\xymatrix{(a,U^\dagger):\Hom_{\Lambda^{\op}}(\Tr T^0\oplus P_0,U^\dagger)\ar[r]&\Hom_{\Lambda^{\op}}(P_1,U^\dagger)}
\end{equation}
is surjective.

Let $Q_1\xto{d} Q_0\to U\to0$ be a minimal projective presentation of $U$.
Let $d':Q'_1\to Q_0$ be a right $(\add\Lambda e')$-approximation of $Q_0$.
Since $T^0=U/\langle e'\rangle U$ by (i), we have a projective presentation
$Q'_1\oplus Q_1\xto{{d'\choose d}} Q_0\to T^0\to0$ of $T^0$.
Thus we have an exact sequence
\[\xymatrix{Q_0^*\ar[r]^(.4){(d'{}^*\ d^*)}&Q'_1{}^*\oplus Q_1^*\ar[r]^{{c'\choose c}}&\Tr T^0\oplus Q\ar[r]&0}\]
for some projective $\Lambda^{\op}$-module $Q$. We have a commutative diagram
\[\xymatrix{
Q_0^*\ar[r]^{d^*}\ar[d]^{d'{}^*}&Q_1^*\ar[r]\ar[d]^{-c}&\Tr U\ar[r]\ar@{=}[d]&0\\
Q'_1{}^*\ar[r]^(.4){c'}&\Tr T^0\oplus Q\ar[r]&\Tr U\ar[r]&0
}\]
of exact sequences. Now we decompose the morphism $c'$ as
\[\xymatrix{c'={a\ \ 0\ \ \choose 0\ 1_{Q''}}:Q'_1{}^*=P_1\oplus Q''\ar[r]&
\Tr T^0\oplus Q=\Tr T^0\oplus P_0\oplus Q'',}\]
where $a$ is in the radical. Then we naturally have an exact sequence
\eqref{another sequence for Tr U}, and clearly we have $P_0\in\proj\Lambda^{\op}$
and $P_1\in\add e'\Lambda$ by our construction.
It remains to show that \eqref{to Tr U} is surjective. We only have to show that the map
\[\xymatrix{(c',U^\dagger):\Hom_{\Lambda^{\op}}(\Tr T^0\oplus Q,U^\dagger)\ar[r]&\Hom_{\Lambda^{\op}}(Q'_1{}^*,U^\dagger)}\]
is surjective.
Take any map $s:Q'_1{}^*\to U^\dagger$. By Proposition \ref{no common summand}(c),
there exists $t:Q_1^*\to U^\dagger$ such that $sd'{}^*=td^*$.
Thus there exists $u:\Tr T^0\oplus Q\to U^\dagger$ such that $s=uc'$ and $t=-uc$,
which shows the assertion.

(iv) First we assume $P_1$ in (iii) is non-zero.
Since $e'\Lambda\in\add T^\dagger_0$ by (ii) and $P_1\in\add e'\Lambda$, we have $P_1\in\add T^\dagger_0$.
Thus there exists a morphism $s:P_1\to T^\dagger_0$ which is not in the radical.
Since \eqref{to Tr U} is surjective, there exists $t:\Tr T^0\oplus P_0\to U^\dagger$ such that $ta=fs$.
Since $f$ is a surjective right $(\add T^\dagger)$-approximation and $P_0$ is projective,
there exists $u:\Tr T^0\oplus P_0\to T^\dagger_0$ such that $t=fu$.
\[\xymatrix{
P_1\ar[r]^(.4)a\ar[d]^s&\Tr T^0\oplus P_0\ar[r]\ar[d]^t\ar@{.>}[dl]^u&\Tr U\ar[r]&0\\
T^\dagger_0\ar[r]_f&U^\dagger\ar[r]&0
}\]
Since $f(s-ua)=0$ and $f$ is right minimal, we have that $s-ua$ is in the radical.
Since $a$ is in the radical, so is $s$, a contradiction.

Consequently, we have $P_1=0$. Thus $\Tr T^0\oplus P_0\simeq\Tr U$ and
$\Tr T^0\simeq \Tr U$. Since $T\in\add T^0$ by our assumption, we have
$\add T_{\rm np}=\add U_{\rm np}$. Since $U>T$, we have
$T_{\rm pr}\in\add U_{\rm pr}$. Thus $U\simeq T\oplus P$ for some projective
$\Lambda$-module $P$.

(v) It remains to consider the case $U\simeq T\oplus P$ for some projective
$\Lambda$-module $P$.

Since $U>T$, we have $\add \Lambda e\subsetneq\add\Lambda e'$.
Take any indecomposable summand $\Lambda e''$ of $\Lambda(e'-e)$
and let $V:=\mu_{\Lambda e''}(T,\Lambda e')$, which has a form
$(T\oplus X,\Lambda(e'-e''))$ with $X$ indecomposable.
Clearly $V>T$ holds. Since $\tau U\in\add\tau(T\oplus X)$ by our assumption
and $\Lambda e\in\add\Lambda(e'-e'')$ by our choice of $e''$, we have
\[\Fac U={}^\perp(\tau U)\cap(\Lambda e)^\perp\supseteq
{}^\perp(\tau(T\oplus X))\cap(\Lambda(e'-e''))^\perp=\Fac V\]
by Corollary \ref{analog of 3 conditions}(c). Thus $U\ge V$ holds.
\qed

\medskip
We end this section with the following application, which is an analog of \cite[Corollary 2.2]{HU2}.

\begin{corollary}\label{connected}
If $\Q(\sttilt\Lambda)$ has a finite connected component $C$, then $\Q(\sttilt\Lambda)=C$.
\end{corollary}

\begin{proof}
Fix $T$ in $C$. Applying Theorem \ref{make closer}(a) to $\Lambda\ge T$, we have a sequence $T=T_0<T_1<T_2<\cdots$ of right mutations of support $\tau$-tilting modules such that $\Lambda\ge T_i$ for any $i$.
Since $C$ is finite, this sequence must be finite. Thus $\Lambda=T_i$ for some $i$, and $\Lambda$ belongs to $C$.
Now we fix any $U\in\sttilt\Lambda$. Applying Theorem \ref{make closer}(b) to $\Lambda\ge U$, we have a sequence $\Lambda=V_0>V_1>V_2>\cdots$ of left mutations of support $\tau$-tilting modules such that $V_i\ge U$ for any $i$.
Since $C$ is finite, this sequence must be finite. Thus $U=V_j$ for some $j$, and $U$ belongs to $C$.
\end{proof}

\section{Connection with silting theory}

Throughout this section, let $\Lambda$ be a finite dimensional algebra
over a field $k$. Any almost complete silting complex has infinitely many complements.
But if we restrict to two-term silting complexes, we get another class of objects extending
the (classical) tilting modules and satisfying the two complement property (Corollary \ref{sptsil5}).
Moreover we will show that there is a bijection between support $\tau$-tilting
$\Lambda$-modules and two-term silting complexes for $\Lambda$, which is of independent interest (Theorem \ref{sptsil4}).
The two-term silting complexes are defined as follows.

\begin{definition}
We call a complex $P=(P^{i},d^i)$ in $\KKb(\proj \Lambda)$ \emph{two-term}\index{two-term complex} if $P^{i}=0$ for all $i\neq 0,-1$.
Clearly $P\in\KKb(\proj \Lambda)$ is two-term\index{two-term complex} if and only if $\Lambda\ge P\ge\Lambda[1]$.
\end{definition}

We denote by $\twosilt\Lambda$\index{2silt@$\twosilt\Lambda$} (respectively, $\twopresilt\Lambda$\index{2presilt@$\twopresilt\Lambda$}) the set of isomorphism classes of basic two-term silting (respectively, presilting) complexes for $\Lambda$.

Clearly any two-term complex is isomorphic to a two-term complex $P=(P^{i},d^i)$ satisfying $d^{-1}\in \rad(P^{-1}, P^{0})$ in $\KKb(\proj \Lambda)$.
Moreover, for any two-term complexes $P$ and $Q$, we have
$\Hom_{\KKb(\proj\Lambda)}(P,Q[i])= 0$ for any $i\neq -1,0,1$.

The aim of this section is to prove the following result.

\begin{theorem}\label{sptsil4}
Let $\Lambda$ be a finite dimensional $k$-algebra.
Then there exists a bijection
\[\twosilt\Lambda\longleftrightarrow\sttilt\Lambda\]
given by $\twosilt\Lambda\ni P\mapsto H^0(P)\in\sttilt\Lambda$
and $\sttilt\Lambda\ni(M,P)\mapsto (P_1\oplus P\xto{(f\ 0)}P_0)\in\twosilt\Lambda$
where $f:P_{1}\rightarrow P_{0}$ is a minimal projective presentation of $M$.
\end{theorem}

The following result is quite useful.

\begin{proposition}\label{silt2.28}
Let $P$ be a two-term presilting complex for $\Lambda$.
\begin{itemize}
\item[(a)] $P$ is a direct summand of a two-term silting complex for $\Lambda$.
\item[(b)] $P$ is a silting complex for $\Lambda$ if and only if $|P|=|\Lambda|$.
\end{itemize}
\end{proposition}

\begin{proof}
(a) This is shown in \cite[Proposition 2.16]{Ai}.

(b) The `only if' part follows from Proposition \ref{number of summands of silting}(a).
We will show the `if' part.
Let $P$ be a two-term presilting complex for $\Lambda$ with $|P|=|\Lambda|$.
By (a), there exists a complex $X$ such that $P\oplus X$ is silting.
Then we have $|P\oplus X| =|\Lambda|=|P|$ by Proposition \ref{number of summands of silting}(a),
so $X$ is in $\add P$. Thus $P$ is silting.
\end{proof}

The following lemma is important.

\begin{lemma}\label{sptsil1}
Let $M,N\in\mod \Lambda$. 
Let $P_{1}\overset{p_{1}}{\rightarrow} P_{0}\overset{p_{0}}{\rightarrow} M\rightarrow 0$ and 
$Q_{1}\overset{q_{1}}{\rightarrow} Q_{0}\overset{q_{0}}{\rightarrow} N\rightarrow 0$
be minimal projective presentations of $M$ and $N$ respectively.
Let $P=(P_{1}\overset{p_{1}}{\rightarrow} P_{0})$ and $Q=(Q_{1}\overset{q_{1}}{\rightarrow} Q_{0})$ be two-term complexes for $\Lambda$.
Then the following conditions are equivalent:
\begin{itemize}
\item[(a)] $\Hom_{\Lambda}(N,\tau M)=0$.
\item[(b)] $\Hom_{\KKb(\proj\Lambda)}(P,Q[1])=0$.
\end{itemize}
In particular, $M$ is a $\tau$-rigid $\Lambda$-module if and only if $P$ is a presilting complex for $\Lambda$.
\end{lemma}

\begin{proof}
The condition (a) is equivalent to the fact that $(p_{1},N):\Hom_{\Lambda}(P_{0},N)\rightarrow \Hom_{\Lambda}(P_{1},N)$ is surjective by Proposition \ref{no common summand}(b).

(a)$\Rightarrow$(b) Any morphism $f\in\Hom_{\KKb(\proj\Lambda)}(P,Q[1])$ is given by some $f\in\Hom_\Lambda(P_1,Q_0)$.
Since $(p_{1},N)$ is surjective, there exists $g:P_{0}\rightarrow N$ such that $q_{0} f=gp_{1}$.
Moreover, since $P_{0}$ is projective, there exists $h_{0}:P_{0}\rightarrow Q_{0}$ such that $q_{0}h_{0}=g$.
Since $q_{0}(f-h_{0}p_{1})=0$, we have $h_{1}:P_{1}\rightarrow Q_{1}$ with $f=q_{1} h_{1}+h_{0}p_{1}$.
\[
\xymatrix{
&0\ar[r]&P_{1}\ar[r]^{p_{1}}\ar[d]_{f}\ar@{.>}[dl]_{h_{1}}&P_{0}\ar[r]^{p_{0}}\ar[d]^g\ar@{.>}[dl]_{h_{0}}&M\ar[r]&0\\
0\ar[r]&Q_{1}\ar[r]_{q_{1}}&Q_{0}\ar[r]_{q_{0}}&N\ar[r]&0.
}
\]
Hence we have $\Hom_{\KKb(\proj\Lambda)}(P,Q[1])=0$.

(b)$\Rightarrow$(a) Take any $f\in\Hom_\Lambda(P_1,N)$.
Since $P_{1}$ is projective, there exists $g:P_{1}\rightarrow Q_{0}$ such that $q_{0} g=f$.
\[
\xymatrix{
&P_{1}\ar[r]^{p_{1}}\ar@{.>}[d]_g\ar[dr]^{f}&P_{0}\\
Q_{1}\ar[r]_{q_{1}}&Q_{0}\ar[r]_{q_{0}}&N \ar[r]&0.
}
\]
Then $g$ gives a morphism $P\to Q[1]$ in $\KKb(\proj\Lambda)$.
Since $\Hom_{\KKb(\proj\Lambda)}(P,Q[1])=0$, there exist $h_{0}:P_{0}\rightarrow Q_{0}$ and $h_{1}:P_{1}\rightarrow Q_{1}$ such that 
$g=q_{1} h_{1}+h_{0}p_{1}$.
Hence we have $f=q_{0}(q_{1} h_{1}+h_{0}p_{1})=q_{0} h_{0}p_{1}$. Therefore $(p_{1},N)$ is surjective.
\end{proof}

We also need the following observation.

\begin{lemma}\label{lemma for support}
Let $P_{1}\overset{p_{1}}{\rightarrow} P_{0}\overset{p_{0}}{\rightarrow} M\rightarrow 0$ be a minimal projective presentation of $M$ in $\mod\Lambda$ and $P:=(P_{1}\overset{p_{1}}{\rightarrow} P_{0})$ be a two-term complex for $\Lambda$.
Then for any $Q$ in $\proj\Lambda$, the following conditions are equivalent.
\begin{itemize}
\item[(a)] $\Hom_\Lambda(Q,M)=0$.
\item[(b)] $\Hom_{\KKb(\proj\Lambda)}(Q,P)=0$.
\end{itemize}
\end{lemma}

\begin{proof}
The proof is left to the reader since it is straightforward.
\end{proof}

The following result shows that silting complexes for $\Lambda$ give support $\tau$-tilting modules.

\begin{proposition}\label{sptsil2}
Let $P=(P_1\overset{d}{\rightarrow}P_{0})$ be a two-term complex for $\Lambda$ and $M:=\Cok d$.
\begin{itemize}
\item[(a)] If $P$ is a silting complex for $\Lambda$ and $d$ is right minimal, then $M$ is a $\tau$-tilting $\Lambda$-module.
\item[(b)] If $P$ is a silting complex for $\Lambda$, then $M$ is a support $\tau$-tilting $\Lambda$-module.
\end{itemize}
\end{proposition}

\begin{proof}
(b) We write $d=(d^{\prime}\ 0):P_1=P'_1\oplus P''_1 \rightarrow P_0$, where $d^{\prime}$ is right minimal. 
Then the sequence $P'_1\overset{d^{\prime}}{\rightarrow}P_0\rightarrow M\rightarrow 0$ is a minimal projective presentation of $M$.
We show that $(M,P_1'')$ is a support $\tau$-tilting pair for $\Lambda$.
Since $P$ is silting, $M$ is a $\tau$-rigid $\Lambda$-module by Lemma \ref{sptsil1}.
On the other hand, since $P$ is silting, we have
$\Hom_{\KKb(\proj\Lambda)}(P''_1,P)=0$. By Lemma \ref{lemma for support}, we have
$\Hom_\Lambda(P''_1,M)=0$. Thus $(M,P''_1)$ is a $\tau$-rigid pair for $\Lambda$.
Since $d'$ is a minimal projective presentation of $M$, we have $|M|=|P'_{1}\xto{d'} P_{0}|$.
Thus we have
\[
|M|+|P''_1|=|P'_{1}\xto{d'} P_{0}|+|P''_1|=|P|,
\]  
which is equal to $|\Lambda|$ by Proposition \ref{number of summands of silting}(a).
Hence $(M,P''_1)$ is a support $\tau$-tilting pair for $\Lambda$.

(a) This is the case $P_1''=0$ in (b).
\end{proof}

The following result shows that support $\tau$-tilting $\Lambda$-modules give silting complexes for $\Lambda$.

\begin{proposition}\label{sptsil3}
Let $P_{1}\xto{d_1} P_{0}\xto{d_0}M\rightarrow 0$ be a minimal projective presentation of $M$ in $\mod\Lambda$.
\begin{itemize}
\item[(a)] If $M$ is a $\tau$-tilting $\Lambda$-module, then $(P_{1}\xto{d_1} P_{0})$ is a silting complex for $\Lambda$.
\item[(b)] If $(M,Q)$ is a support $\tau$-tilting pair for $\Lambda$,
then $P_{1} \oplus Q\xrightarrow{(d_1\ 0)} P_{0}$ is a silting complex for $\Lambda$.
\end{itemize}
\end{proposition}

\begin{proof}
(b) We know that $(P_{1}\xto{d_1} P_{0})$ is a presilting complex for $\Lambda$ by Lemma \ref{sptsil1}.
Let $P:=(P_{1}\oplus Q\xto{(d_1\ 0)} P_{0})$.
By Lemmas \ref{sptsil1} and \ref{lemma for support}, we have that $P$ is a presilting complex for $\Lambda$.
Since $d_1$ is a minimal projective presentation, we have $|P_{1}\xto{d_1}P_{0}|=|M|$.
Moreover, since $(M,Q)$ is a support $\tau$-tilting pair for $\Lambda$,
we have $|M|+|Q|=|\Lambda|$. Thus we have 
\[
|P|=|P_{1}\xto{d_1}P_{0}|+|Q|=|M|+|Q|=|\Lambda|.  
\]
Hence $P$ is a silting complex for $\Lambda$ by Proposition \ref{silt2.28}(b).

(a) This is the case $Q=0$ in (b).
\end{proof}

Now Theorem \ref{sptsil4} follows from Propositions \ref{sptsil2} and \ref{sptsil3}.
\qed

\medskip
We give some applications of Theorem \ref{sptsil4}.

\begin{corollary}\label{sptsil5}
Let $\Lambda$ be a finite dimensional $k$-algebra.
\begin{itemize}
\item[(a)] Any basic two-term presilting complex $P$ for $\Lambda$ with $|P|=|\Lambda|-1$ is a direct summand of exactly two basic two-term silting complexes for $\Lambda$.
\item[(b)] Let $P,Q\in\twosilt\Lambda$. Then $P$ and $Q$ have  all but one indecomposable direct summand in common if and only if $P$ is a left or right mutation of $Q$.
\end{itemize}
\end{corollary}

\begin{proof}
(a) This follows from Theorems \ref{2 complements} and \ref{sptsil4}.

(b) This is immediate from (a).
\end{proof}

Now we define $\Q(\twosilt\Lambda)$\index{3q2silt@$\Q(\twosilt\Lambda)$} as the full subquiver of $\Q(\silt\Lambda)$ with vertices corresponding to two-term silting complexes for $\Lambda$.

\begin{corollary}\label{sptsil6}
The bijection in Theorem \ref{sptsil4} is an isomorphism of the partially ordered sets.
In particular, it induces an isomorphism between the two-term silting quiver $\Q(\twosilt\Lambda)$ and the support $\tau$-tilting quiver $\Q(\sttilt\Lambda)$.
\end{corollary}

\begin{proof}
Let $(M,\Lambda e)$ and $(N,\Lambda f)$ be support $\tau$-tilting pairs for $\Lambda$.
Let $P:=(P_{1}{\rightarrow} P_{0})$ and $Q:=(Q_{1}{\rightarrow}Q_{0})$ be minimal projective presentations of $M$ and $N$ respectively.
We only have to show that $M\ge N$ if and only if $\Hom_{\KKb(\proj\Lambda)}(P\oplus \Lambda e[1],(Q\oplus\Lambda f[1])[1])=0$.

We know that $M\ge N$ if and only if $\Hom_\Lambda(N,\tau M)=0$ and $\Lambda e\in\add\Lambda f$ by Lemma \ref{partial order of tautilting}.
Moreover $\Hom_\Lambda(N,\tau M)=0$ if and only if $\Hom_{\KKb(\proj\Lambda)}(P,Q[1])=0$ by by Lemma \ref{sptsil1}.
On the other hand $\Lambda e\in\add\Lambda f$ if and only if $\Hom_\Lambda(\Lambda e,N)=0$ since $N$ is a sincere $(\Lambda/\langle f\rangle)$-module.
Thus $\Lambda e\in\add\Lambda f$ is equivalent to $\Hom_{\KKb(\proj\Lambda)}(\Lambda e,Q)=0$ by Lemma \ref{lemma for support}.
Consequently $M\ge N$ if and only if $\Hom_{\KKb(\proj\Lambda)}(P\oplus \Lambda e[1],Q[1])=0$,
and this is equivalent to $\Hom_{\KKb(\proj\Lambda)}(P\oplus \Lambda e[1],(Q\oplus\Lambda f[1])[1])=0$ since $\Hom_{\KKb(\proj\Lambda)}(P\oplus \Lambda e[1],\Lambda f[2])=0$ is automatic.
Thus the assertion follows.
\end{proof}

Immediately we have the following application.

\begin{corollary}\label{connected2}
If $\Q(\twosilt\Lambda)$ has a finite connected component $C$, then $\Q(\twosilt\Lambda)=C$.
\end{corollary}

\begin{proof}
This is immediate from Corollaries \ref{connected} and \ref{sptsil6}.
\end{proof}

Note also that Theorem \ref{sptsil4} and Corollary \ref{sptsil6} give an alternative proof of
Theorem \ref{make closer} since the corresponding property for two-term silting
complexes holds by \cite[Proposition 2.36]{AI}.

\section{Connection with cluster-tilting theory}\label{sec3}

Let $\CC$ be a Hom-finite Krull-Schmidt 2-Calabi-Yau (2-CY for short) triangulated category
(for example, the cluster category $\CC_Q$ associated with a finite acyclic quiver $Q$ \cite{BMRRT}). 
We shall assume that our category $\CC$ has a cluster-tilting object $T$.
Associated with $T$, we have by definition the 2-CY-tilted algebra $\Lambda=\End_{\CC}(T)^{\op}$, whose module category is closely connected with the 2-CY-category $\CC$. In particular, there is an equivalence of categories \cite{BMR1,KR}:
\begin{equation}\label{bar functor}
\functor{(-)}:=\Hom_{\CC}(T,-):\CC/[T[1]]\to\mod\Lambda.
\end{equation}\index{$\functor{(-)}$}
In this section we investigate this relationship more closely by giving a bijection between cluster-tilting objects in $\CC$ and support $\tau$-tilting 
$\Lambda$-modules (Theorem \ref{bijection between CT and PT}). 
This was the starting point for the theory of $\tau$-rigid and $\tau$-tilting modules.
As an application, we give a proof of some known results for cluster-tilting objects (Corollary \ref{application to CT}). Also we give a direct connection between
cluster-tilting objects in $\CC$ and two-term silting complexes for $\Lambda$ (Theorem \ref{ctsilt3}).
There is an induced isomorphism between the associated graphs (Corollary \ref{ctsilt4}).


\subsection{Support $\tau$-tilting modules and cluster-tilting objects}

In this subsection we show that there is a close relationship between
the cluster-tilting objects in $\CC$ and support $\tau$-tilting $\Lambda$-modules.
We use this to apply our main Theorem \ref{main theorem 1}
to get a new proof of the fact that almost complete cluster-tilting objects have exactly two complements,
and of the fact that all maximal rigid objects are cluster-tilting, as first proved in \cite{IY} and \cite{ZZ}, respectively.

We denote by $\iso\CC$\index{3iso@$\iso\CC$} the set of isomorphism classes of objects in a category $\CC$.
From our equivalence \eqref{bar functor}, we have a bijection
\begin{equation*}
\widetilde{(-)}:\iso\CC\longleftrightarrow\iso(\mod\Lambda)\times\iso(\proj\Lambda)
\end{equation*}\index{$\widetilde{(-)}$}
given by $X=X'\oplus X''\mapsto\widetilde{X}:=(\functor{X'},\functor{X''[-1]})$, where $X''$ is a maximal direct summand of $X$ which belongs to $\add T[1]$.
We denote by $\rigid\CC$\index{3rigid@$\rigid\CC$} (respectively, $\mrigid\CC$\index{3mrigid@$\mrigid\CC$}) the set of isomorphism classes of basic rigid (respectively, maximal rigid) objects in $\CC$,
and by $\ctilt_T\CC$\index{3ctiltt@$\ctilt_T\CC$} the set of isomorphism classes of basic cluster-tilting objects in $\CC$ which do not have non-zero direct summands in $\add T[1]$.

Our main result in this section is the following.

\begin{theorem}\label{bijection between CT and PT}
The bijection $\widetilde{(-)}$ induces bijections
\begin{eqnarray*}
\rigid\CC\longleftrightarrow\trigid\Lambda,\ \ \ 
\ctilt\CC\longleftrightarrow\sttilt\Lambda\ \ \ \mbox{and}\ \ \ 
\ctilt_T\CC\longleftrightarrow\ttilt\Lambda.
\end{eqnarray*}
Moreover we have $\ctilt\CC=\mrigid\CC=\{U\in\rigid\CC\ |\ |U|=|T|\}$.
\end{theorem}

We start with the following easy observation (see \cite{KR}). 

\begin{lemma}\label{ctsilt1}
The functor $\functor{(-)}$ induces an equivalence of categories between $\add T$ (respectively, $\add T[2]$) and $\proj \Lambda$ (respectively, $\inj \Lambda$).
Moreover we have an isomorphism $\functor{(-)}\circ[2]\simeq\nu\circ\functor{(-)}:\add T\to\inj \Lambda$ of functors.
\end{lemma}

Now we express $\Ext^1_{\CC}(X,Y)$ in terms of the images $\overline{X}$
and $\overline{Y}$ in our fixed 2-CY tilted algebra $\Lambda$.
We let
\[\langle X,Y\rangle_\Lambda=\langle X,Y\rangle:=\dim_k\Hom_\Lambda(X,Y).\]\index{$\langle -,-\rangle$}

\begin{proposition}\label{prop3.1}
Let $X$ and $Y$ be objects in $\CC$. Assume that
there are no nonzero indecomposable direct summands of $T[1]$ for $X$ and $Y$.
\begin{itemize}
\item[(a)] We have $\functor{X[1]}\simeq\tau\functor{X}$ and $\functor{Y[1]}\simeq\tau\functor{Y}$ as $\Lambda$-modules.
\item[(b)] We have an exact sequence
$$0\to D\Hom_{\Lambda}(\functor{Y},\tau\functor{X})\to\Ext^1_{\CC}(X,Y)\to\Hom_{\Lambda}(\functor{X},\tau\functor{Y})\to 0.$$
\item[(c)] $\dim\Ext_{\CC}^1(X,Y)=\langle\functor{X},\tau\functor{Y}\rangle_\Lambda+\langle\functor{Y},\tau\functor{X}\rangle_\Lambda$.
\end{itemize}
\end{proposition}

\begin{proof}
(a) This can be shown as in the proof of \cite[Proposition 3.2]{BMR1}. Here we give a direct proof. Take a triangle
\begin{equation}\label{T resolution}
\xymatrix{T_1\ar[r]^{g}& T_0\ar[r]^{f}& X\ar[r]& T_1[1]}
\end{equation}
with a minimal right $(\add T)$-approximation $f$ and $T_0,T_1\in\add T$.
Applying $\functor{(\ )}$ to \eqref{T resolution}, we have an exact sequence
\begin{equation}\label{min proj resol}
\xymatrix{\functor{T_1}\ar[r]^{\functor{g}}&\functor{T_0}\ar[r]^{\functor{f}}&\functor{X}\ar[r]&0.}
\end{equation}
This gives a minimal projective presentation of $\functor{X}$ since $X$ has no nonzero indecomposable direct summands of $T[1]$.
Applying the Nakayama functor to \eqref{min proj resol} and $\Hom_{\CC}(T,-)$ to \eqref{T resolution} and comparing them by Lemma \ref{ctsilt1}, we have the following commutative diagram of exact sequences:
\[\xymatrix{
0\ar[r]&\tau\functor{X}\ar[r]&\nu\functor{T_1}\ar[r]^{\nu\functor{g}}\ar[d]^\wr&\nu\functor{T_0}\ar[d]^\wr\\
0=\functor{T_0[1]}\ar[r]&\functor{X[1]}\ar[r]&\functor{T_1[2]}\ar[r]^{\functor{g[2]}}&\functor{T_0[2]}.
}\]
Thus we have $\tau\functor{X}\simeq\functor{X[1]}$.

(b) We have an exact sequence
\[0\to[T[1]](X,Y[1])\to\Hom_{\CC}(X,Y[1])\to\Hom_{\CC/[T[1]]}(X,Y[1])\to0,\]
where $[T[1]]$ is the ideal of $\CC$ consisting of morphisms which factor through $\add T[1]$.
We have a functorial isomorphism
\begin{equation}\label{X,Y[1]}
\Hom_{\CC/[T[1]]}(X,Y[1])\simeq\Hom_\Lambda(\functor{X},\functor{Y[1]})\stackrel{{\rm (a)}}{\simeq}\Hom_\Lambda(\functor{X},\tau\functor{Y}).
\end{equation}
On the other hand, the first of following functorial isomorphism was given in \cite[3.3]{P}.
\[[T[1]](X,Y[1])\simeq D\Hom_{\CC/[T[1]]}(Y,X[1])\stackrel{\eqref{X,Y[1]}}{\simeq}D\Hom_\Lambda(\functor{Y},\tau\functor{X}).\]
Thus the assertion follows.

(c) This is immediate from (b).
\end{proof}

We now consider the general case, where we allow indecomposable direct summands from $T[1]$ in $X$ or $Y$.

\begin{proposition}\label{thm3.3}
Let $X=X'\oplus X''$ and $Y=Y'\oplus Y''$ be objects in $\CC$
such that $X''$ and $Y''$ are the maximal direct summands of $X$ and $Y$ respectively,
which belong to $\add T[1]$. Then
\begin{eqnarray*}
\dim\Ext^1_{\CC}(X,Y)&=&\langle\functor{X'},\tau\functor{Y'}\rangle_\Lambda+\langle\functor{Y'},\tau\functor{X'}\rangle_\Lambda+
\langle\functor{X''[-1]},\functor{Y'}\rangle_\Lambda+\langle\functor{Y''[-1]},\functor{X'}\rangle_\Lambda.
\end{eqnarray*}
\end{proposition}

\begin{proof}
Since $\Ext^1_{\CC}(X'',Y'')=0$, we have
\[\dim\Ext^1_{\CC}(X,Y)=\dim\Ext^1_{\CC}(X',Y')+\dim\Ext^1_{\CC}(X'',Y')+\dim\Ext^1_{\CC}(X',Y'').\]
By Proposition \ref{prop3.1}, the first term equals $\langle\functor{X'},\tau\functor{Y'}\rangle_\Lambda+\langle\functor{Y'},\tau\functor{X'}\rangle_\Lambda$.
Clearly the second term equals $\langle\functor{X''[-1]},\functor{Y'}\rangle_\Lambda$,
and the third term equals $\langle\functor{Y''[-1]},\functor{X'}\rangle_\Lambda$.
\end{proof}

Now we are ready to prove Theorem \ref{bijection between CT and PT}.

By Proposition \ref{thm3.3}, we have that $X$ is rigid if and only if
$\widetilde{X}$ is a $\tau$-rigid pair for $\Lambda$.
Thus we have bijections $\rigid\CC\leftrightarrow\trigid\Lambda$, which induces a bijection $\mrigid\CC\leftrightarrow\sttilt\Lambda$ by Corollary \ref{analog of 3 conditions}(a)$\Leftrightarrow$(b).

On the other hand we show that a bijection $\ctilt\CC\leftrightarrow\sttilt\Lambda$ is induced. Since $\ctilt\CC\subseteq\mrigid\CC$, we only have to show that
any $X\in\rigid\CC$ satisfying that $\widetilde{X}$ is a support $\tau$-tilting pair for $\Lambda$ is a cluster-tilting object in $\CC$.
Assume that $Y\in\CC$ satisfies $\Ext^1_{\CC}(X,Y)=0$. By Proposition \ref{thm3.3},
we have $\Hom_\Lambda(\functor{X'},\tau\functor{Y'})=0$, $\Hom_\Lambda(\functor{Y'},\tau\functor{X'})=0$, $\Hom_\Lambda(\functor{X''[-1]},\functor{Y'})=0$ and
$\Hom_\Lambda(\functor{Y''[-1]},\functor{X'})=0$.
By the first 3 equalities, we have $\functor{Y'}\in\add\functor{X'}$
by Corollary \ref{analog of 3 conditions}(a)$\Leftrightarrow$(d).
By the last equality we have $\functor{Y''[-1]}\in\add\functor{X''[-1]}$.
Thus $Y\in\add X$ holds, which shows that $X$ is a cluster-tilting object in $\CC$.

The remaining statements follow immediately.
\qed

\medskip
Now we recover the following results in \cite{IY} and \cite{ZZ}.

\begin{corollary}\label{application to CT}
Let $\CC$ be a 2-CY triangulated category with a cluster-tilting object $T$.
\begin{itemize}
\item[(a)] \cite{IY} Any basic almost complete cluster-tilting object is a direct summand of exactly two basic cluster-tilting objects.
In particular, $T$ is a mutation of $V$ if and only if $T$ and $V$ have all but one indecomposable direct summand in common.
\item[(b)] \cite{ZZ} An object $X$ in $\CC$ is cluster-tilting if and only if it is maximal rigid if and only if it is rigid and $|X|=|T|$.
\end{itemize}
\end{corollary}

\begin{proof}
(a) This is immediate from the bijections given in Theorem \ref{bijection between CT and PT}
and the corresponding result for support $\tau$-tilting pairs given in Theorem \ref{2 complements}.

(b) This is the last equality in Theorem \ref{bijection between CT and PT}.
\end{proof}

Connections between cluster-tilting objects in $\CC$ and tilting $\Lambda$-modules have been investigated in \cite{Smi,FL}.
It was shown that a tilting $\Lambda$-module always comes from a cluster-tilting object in $\CC$, but the image of a cluster-tilting object is not always a tilting $\Lambda$-module.
This is explained by Theorem \ref{bijection between CT and PT} asserting that the $\Lambda$-modules corresponding to the cluster-tilting objects of $\CC$ are the support $\tau$-tilting $\Lambda$-modules, which are not necessarily tilting $\Lambda$-modules.

\subsection{Two-term silting complexes and cluster-tilting objects}

Throughout this section, let $\CC$ be a 2-CY category with a cluster-tilting object $T$. 
Fix a cluster-tilting object $T\in \CC$.
Let $\Lambda:=\End_{\CC}(T)^{\op}$ and let $\KKb(\proj \Lambda)$ be the homotopy category of bounded complexes of finitely generated projective $\Lambda$-modules.
In this section, we shall show that there is a bijection between cluster-tilting objects in $\CC$ and two-term silting complexes for $\Lambda$ and that the mutations are compatible with each other.
  
The following result will be useful, where we denote by $\KKtwo(\proj\Lambda)$\index{3k2@$\KKtwo(\proj\Lambda)$} the full subcategory of $\KKb(\proj\Lambda)$ consisting of two-term complexes for $\Lambda$.

\begin{proposition}\label{ctsilt2}
There exists a bijection
\[\iso\CC\longleftrightarrow\iso(\KKtwo(\proj\Lambda))\]
which preserves the number of non-isomorphic indecomposable direct summands.
\end{proposition}

\begin{proof}
For any object $U\in \CC$, there exists a triangle
\[
\xymatrix{
T_{1}\ar[r]^{g}&T_{0}\ar[r]^{f}&U\ar[r]&T_{1}[1]
}
\]
where $T_{1},T_{0}\in \add T$ and $f$ is a minimal right $(\add T)$-approximation. 
By Lemma \ref{ctsilt1}, we have a two-term complex $\functor{T_{1}}\xto{\functor{g}}\functor{T_{0}}$ in $\KKb(\proj\Lambda)$.

Conversely, let $P_{1}\overset{d}{\rightarrow}P_{0}$ be a two-term complex for $\Lambda$.
By Lemma \ref{ctsilt1}, there exists a morphism $g:T_{1}\rightarrow T_{0}$ in $\add T$ such that $\functor{g}=d$.
Taking the cone of $g$, we have an object $U$ in $\CC$.
Then we can easily check that the correspondence gives a bijection and preserves the number of non-isomorphic indecomposable direct summands. 
\end{proof}

Using this, we get the desired correspondence.

\begin{theorem}\label{ctsilt3}
The bijection in Proposition \ref{ctsilt2} induces bijections
\begin{eqnarray*}
\rigid\CC\longleftrightarrow\twopresilt\Lambda\ \ \ \mbox{ and }\ \ \ \ctilt\CC\longleftrightarrow\twosilt\Lambda.
\end{eqnarray*}
\end{theorem}

\begin{proof}
(i) For any rigid object $U\in \CC$, we have a triangle
\[
\xymatrix{T_{1}\ar[r]^{g}&T_{0}\ar[r]^{f}&U\ar[r]^{h}&T_{1}[1]}
\]
where $T_{1},T_{0}\in \add T$ and $f$ is a minimal right $(\add T)$-approximation. 
Let $a: T_{1}\rightarrow T_{0}$ be an arbitrary morphism in $\CC$.
Since $U$ is rigid, we have $fah[-1]=0$. Thus we have a commutative diagram
\[
\xymatrix{
U[-1]\ar[r]^{h[-1]}\ar[d]&T_{1}\ar[r]^{g}\ar[d]^{a}&T_{0}\ar[r]^{f}\ar[d]^{b}&U\ar[d]\\
T_{1}\ar[r]^{g}&T_{0}\ar[r]^{f}&U\ar[r]^{h}&T_{1}[1]
}
\]
of triangles in $\CC$. Since $hb=0$, there exists $k_{0}:T_{0}\rightarrow T_{0}$
such that $b=fk_{0}$. Since $f(a-k_{0}g)=0$, there exists
$k_{1}:T_{1}\rightarrow T_{1}$ such that $gk_{1}=a-k_{0}g$. Therefore we have 
\[
\Hom_{\KKb(\proj\Lambda)}((\functor{T_{1}}\xrightarrow{\functor{g}}\functor{T_{0}}), (\functor{T_{1}}\xrightarrow{\functor{g}}\functor{T_{0}})[1])=0. 
\]
Thus $\functor{T_{1}}\xrightarrow{\functor{g}}\functor{T_{0}}$ is a presilting complex for $\Lambda$.

(ii) Let $P:=(P_{1}\overset{d}{\rightarrow} P_{0})$ be a two-term presilting complex
for $\Lambda$. There exists a unique $g:T_{1}\rightarrow T_{0}$ in $\add T$ such that
$\functor{g}=d$. We consider a triangle
\[
\xymatrix{T_{1}\ar[r]^{g}&T_{0}\ar[r]^{f}&U\ar[r]^{h}&T_{1}[1]}
\]
in $\CC$. We take a morphism $a:U\rightarrow U[1]$ in $\CC$. Then we have the commutative diagram
\[
\xymatrix{
T_{1}\ar[r]^{g}\ar[d]&T_{0}\ar[r]\ar[d]^{h[1]af}&0\ \ar[d]\\
0\ar[r]&T_{1}[2]\ar[r]^{g[2]}&T_{0}[2].
}
\]
Applying $\functor{(-)}$, we have a commutative diagram
\[
\xymatrix{
P_{1}\ar[r]^{d}\ar[d]&P_{0}\ar[r]\ar[d]^{\functor{h[1]af}}&0\ar[d]\\
0\ar[r]&\nu P_{1}\ar[r]^{\nu d}&\nu P_{0}.
}
\]
Thus we have a morphism $P\to\nu P[-1]$ in $\KKb(\proj\Lambda)$.
Since $P$ is a presilting complex for $\Lambda$, we have 
\[
\Hom_{\KKb(\proj\Lambda)}(P,\nu P[-1])\simeq \kD \Hom_{\KKb(\proj\Lambda)}(P[-1],P)=0.
\]
Therefore $\functor{h[1]af}=0$, and the morphism $h[1]af$ factors through $\add T[1]$.
Hence we have $h[1]af=0$. Thus we have a commutative diagram
\[\xymatrix{
T_1\ar[r]^g&T_{0}\ar[r]^{f}\ar[d]^{a_{0}}&U\ar[r]^{h}\ar[d]^{a}&T_{1}[1]\\
T_{1}[1]\ar[r]^{g[1]}&T_{0}[1]\ar[r]^{f[1]}&U[1]\ar[r]^{h[1]}&T_{1}[2].}
\]
Since $T_{0}\in\add T$, we have $a_{0}=0$.
Thus $af=0$, so there exists $\varphi:T_{1}[1]\rightarrow U[1]$ such that $a=\varphi h$.
Since $T_{1}\in\add T$, we have $h[1]\varphi=0$.
Thus there exists $b:T_{1}[1]\rightarrow T_{0}[1]$ such that $\varphi=f[1]b$.
Consequently, we have commutative diagrams
\[\xymatrix{0\ar[r]\ar[d]&T_1\ar[r]^g\ar[d]^{b[-1]}&T_0\ar[d]\\
T_1\ar[r]^g&T_0\ar[r]&0}\ \ \ \ \ 
\xymatrix{0\ar[r]\ar[d]&P_1\ar[r]^d\ar[d]^{\functor{b[-1]}}&P_0\ar[d]\\
P_1\ar[r]^d&P_0\ar[r]&0}
\]
Since $P$ is a presilting complex for $\Lambda$,
there exist $s:T_{0}[1]\rightarrow T_{0}[1]$ and $t:T_{1}[1]\rightarrow T_{1}[1]$ such that $b=sg[1]+g[1]t$. Therefore we have 
\[
a=\varphi h =f[1]bh=f[1]sg[1]h+f[1]g[1]th=0. 
\]
Hence $\Hom_{\CC}(U,U[1])=0$, that is, $U$ is rigid, and the claim follows.
\end{proof}

\begin{corollary}\label{ctsilt4}
The bijections in Theorems \ref{sptsil4} and \ref{ctsilt3} induce isomorphisms of the following graphs.
\begin{itemize}
\item[(a)] The underlying graph of the support $\tau$-tilting quiver $\Q(\sttilt\Lambda)$ of $\Lambda$.
\item[(b)] The underlying graph of the two-term silting quiver $\Q(\twosilt\Lambda)$ of $\Lambda$.
\item[(c)] The cluster-tilting graph $\G(\ctilt\CC)$ of $\CC$.
\end{itemize}
\end{corollary}

\begin{proof}
(a) and (b) are the same by Corollary \ref{sptsil6}.

We show that (b) and (c) are the same.
Let $U$ and $V$ be cluster-tilting objects in $\CC$.
Let $P$ and $Q$ be the two-term silting complexes for $\Lambda$ corresponding respectively to $U$ and $V$ by Theorem \ref{ctsilt3}.
By Corollary \ref{application to CT}(a) the following conditions are equivalent:
\begin{enumerate}
\item[(a)] There exists an edge between $U$ and $V$ in the exchange graph.
\item[(b)] $U$ and $V$ have all but one indecomposable direct summand in common. 
\end{enumerate}
Clearly (b) is equivalent to the following condition:
\begin{enumerate}
\item[(c)] $P$ and $Q$ have all but one indecomposable direct summand in common.
\end{enumerate}
Now (c) is equivalent to the following condition by Corollary \ref{sptsil5}(b).
\begin{enumerate}
\item[(d)] There exists an edge between $P$ and $Q$ in the underlying graph of the silting quiver.
\end{enumerate}
Therefore the exchange graph of $\CC$ and the underlying graph of the silting full subquiver consisting of two-term complexes for $\Lambda$ coincide.
\end{proof}

We end this section with the following application.

\begin{corollary}\label{ctsilt5}
If $\G(\ctilt\CC)$ has a finite connected component $C$, then $\G(\ctilt\CC)=C$.
\end{corollary}

\begin{proof}
This is immediate from Corollaries \ref{connected} and \ref{ctsilt4}.
\end{proof}


\section{Numerical invariants}\label{sec4}

In this section, we introduce $g$-vectors following \cite{AR3} and \cite{DK}.
We show that $g$-vectors of indecomposable direct summands of support $\tau$-tilting modules form a basis of the Grothendieck group (Theorem \ref{g-vectors}).
Moreover we observe that non-isomorphic $\tau$-rigid pairs have different $g$-vectors (Theorem \ref{g-vectors are complete invariant}).
In \cite{DWZ} the authors defined what they called $E$-invariants of finite dimensional
decorated representations of Jacobian algebras, and used this to solve several
conjectures from \cite{FZ}. In the case of finite dimensional Jacobian algebras they showed that the
$E$-invariants were given by formulas which we were led to in section 4.1, by considering
$\dim_k\Ext^1_{\CC}(T,T)$ for a cluster-tilting object $T$ in $\CC$.
We here consider $E$-invariants for any finite dimensional algebra, using the
same formula, and show that they can be expressed in terms of homomorphism spaces,
dimension vectors and $g$-vectors. We give some further 
results on the case of 2-CY tilted algebras, including a comparison for neighbouring
2-CY tilted algebras (Theorem \ref{thm4.3}).

In the rest of this paper we assume that our base field $k$ is algebraically closed. Let $\Lambda$ be a finite dimensional $k$-algebra.

\subsection{$g$-vectors and $E$-invariants for finite dimensional algebras}

Recall from \cite{DK} that the $g$-vectors are defined as follows:
Let $K_0(\proj\Lambda)$ be the Grothendieck group of the additive category $\proj\Lambda$.
Then the isomorphism classes $P(1),\dots,P(n)$ of indecomposable projective $\Lambda$-modules form a basis of $K_0(\proj\Lambda)$.
Consider $M$ in $\mod\Lambda$ and let
\[\xymatrix{P_1\ar[r]& P_0\ar[r]& M\ar[r]&0}\]
be its minimal projective presentation in $\mod\Lambda$. Then we write
\[P_0-P_1=\sum_{i=1}^{n}g_i^MP(i),\]
where by definition $g^M=(g_1^M,\dots,g_n^M)$\index{3g@$g^{M}$} is the \emph{$g$-vector}\index{gvector@$g$-vector!of a module} of $M$.
The element $P_0-P_1$ is also called an \emph{index}\index{index!of a module} of $M$, which
was investigated in \cite{AR3}, in connection with studying modules determined
by their composition factors, and in \cite{DK}.

Another useful vector associated with $M$ is the dimension vector $c^M=(c^M_1,\dots,c_n^M)$\index{3c@$c^{M}$}.
Denote by $S(i)$ the simple top of $P(i)$. Then $c_i^M$ is by definition the multiplicity of the simple module $S(i)$ as composition factor of $M$.
This vector has played an important role in cluster theory for the acyclic case, since the denominators of cluster variables are determined by dimension vectors of indecomposable rigid modules over path algebras \cite{BMRT,CK}.
Now this result is not true in general \cite{BMR2}.

We have the following result on $g$-vectors of support $\tau$-tilting modules.

\begin{theorem}\label{g-vectors}
Let $(M,P)$ be a support $\tau$-tilting pair for $\Lambda$ with $M=\bigoplus_{i=1}^\ell M_i$ and $P=\bigoplus_{i=\ell+1}^nP_i$ with $M_i$ and $P_i$ indecomposable.
Then $g^{M_1},\cdots,g^{M_\ell},g^{P_{\ell+1}},\cdots,g^{P_n}$ form a basis of the Grothendieck group $K_0(\proj\Lambda)$.
\end{theorem}

\begin{proof}
By Theorem \ref{sptsil4}, we have a corresponding silting complex 
$Q=\bigoplus_{i=1}^nQ_i$ for $\Lambda$ with indecomposable $Q_i$, where the vectors $g^{M_1},\cdots,g^{M_\ell},g^{P_{\ell+1}},\cdots,g^{P_n}$
are exactly the classes of $Q_1,\cdots,Q_n$ in the Grothendieck group
$K_0(\KKb(\proj\Lambda))=K_0(\proj\Lambda)$.
By Proposition \ref{number of summands of silting}(b), we have the assertion.
\end{proof}

This gives a result below due to Dehy-Keller. Recall that for a cluster-tilting object $T\in\CC$ and an object $X\in\CC$, there exists
a triangle
\[T''\to T'\to X\to T''[1]\]
in $\CC$ with $T',T''\in\add T$. We call ${\rm ind}_T(X):=T'-T''\in K_0(\add T)$\index{3ind@${\rm ind}_T(X)$} the \emph{index}\index{index!of an object} of $X$.

\begin{corollary}\cite[Theorem 2.4]{DK}
Let $\CC$ be a 2-CY triangulated category, and $T$ and $U=\bigoplus_{i=1}^nU_i$ be basic cluster-tilting objects with $U_i$ indecomposable.
Then the indices ${\rm ind}_T(U_1),\cdots,{\rm ind}_T(U_n)$ form a basis of the Grothendieck group $K_0(\add T)$ of the additive category $\add T$.
\end{corollary}

\begin{proof}
We can assume that $U_i\notin\add T[1]$ for $1\le i\le\ell$, and $U_i\in\add T[1]$ for $\ell+1\le i\le n$.
Then $(\bigoplus_{i=1}^\ell\overline{U_i},\bigoplus_{i=\ell+1}^n\overline{U_i[-1]})$ is a support $\tau$-tilting pair for $\Lambda$ by Theorem \ref{bijection between CT and PT}.
The equivalence $\Hom_{\CC}(T,-):\add T\to\proj\Lambda$ gives an isomorphism $K_0(\add T)\simeq K_0(\proj\Lambda)$.
This sends ${\rm ind}_T(U_i)$ to $g^{\overline{U_i}}$ for $1\le i\le\ell$, and to $-g^{\overline{U_i[-1]}}$ for $\ell+1\le i\le n$. Thus the assertion follows from Theorem \ref{g-vectors}.
\end{proof}

Now we consider a pair $M=(X,P)$ of a $\Lambda$-module $X$ and a projective
$\Lambda$-module $P$. We regard a $\Lambda$-module $X$ as a pair $(X,0)$.
For such pairs $M=(X,P)$ and $N=(Y,Q)$, let
\begin{eqnarray*}
g^M&:=&g^X-g^P,\\
E'_\Lambda(M,N)&:=&\langle X,\tau Y\rangle+\langle P,Y\rangle,\\
E_\Lambda(M,N)&:=&E_\Lambda'(M,N)+E_\Lambda'(N,M),\\
E_\Lambda(M)&:=&E_\Lambda(M,M).
\end{eqnarray*}\index{3g@$g^{M}$}\index{3e@$E'_\Lambda(M,N)$, $E_\Lambda(M,N)$, $E_\Lambda(M)$}
We call $g^M$ the \emph{$g$-vector}\index{gvector@$g$-vector!of a pair} of $M$, and $E_\Lambda(M,N)$ the \emph{$E$-invariant}\index{einvariant@$E$-invariant} of $M$ and $N$.
Clearly a pair $(M,0)$ is $\tau$-rigid if and only if $E_\Lambda(M)=0$.

There is the following relationship between $E$-invariants and $g$-vectors,
where we denote by $a\cdot b$ the standard inner product $\sum_{i=1}^na_ib_i$ for vectors $a=(a_1,\cdots,a_n)$ and $b=(b_1,\cdots,b_n)$.

\begin{proposition}\label{prop4.1}
Let $\Lambda$ be a finite dimensional algebra, and let $X$ and $Y$ be in $\mod\Lambda$. Then we have the following.
\begin{eqnarray*}
E_\Lambda'(X,Y)&=&\langle Y,X\rangle-g^Y\cdot c^X,\\
E_\Lambda(X,Y)&=&\langle Y,X\rangle+\langle X,Y\rangle-g^Y\cdot c^X-g^X\cdot c^Y,\\
E_\Lambda(X)&=&2(\langle X,X\rangle-g^X\cdot c^X).
\end{eqnarray*}
\end{proposition}

\begin{proof}
We only have to show the first equality.
Since $P_0-P_1=\sum_{i=1}^ng_i^YP(i)$, then $\langle P_0,X\rangle-\langle P_1,X\rangle=g^Y\cdot c^X$.
By Proposition \eqref{no common summand}(a), we have 
\[
E'_\Lambda(X,Y)=\langle X,\tau Y\rangle=\langle Y,X\rangle+\langle P_1,X\rangle-\langle P_0,X\rangle
=\langle Y,X\rangle-g^Y\cdot c^X.
\]
\end{proof}

The following more general description of $E$-invariants is also clear.

\begin{proposition}\label{thm4.2}
For any pair $M=(X,P)$ and $N=(Y,Q)$, we have
$$E_\Lambda(M,N)=\langle Y,X\rangle+\langle X,Y\rangle-g^M\cdot c^Y-g^N\cdot c^X.$$
\end{proposition}

We end this subsection with the following analog of \cite[Theorem 2.3]{DK}, which was also observed by Plamondon.

\begin{theorem}\label{g-vectors are complete invariant}
The map $M\mapsto g^M$ gives an injection from the set of isomorphism classes of $\tau$-rigid pairs for $\Lambda$ to $K_0(\proj\Lambda)$.
\end{theorem}

\begin{proof}
The proof is based on Propositions \ref{no common summand}(c) and \ref{no common summand 2}, and is the same as that of \cite[Theorem 2.3]{DK}.
\end{proof}

\subsection{$E$-invariants for $2$-CY tilted algebras}

In the rest of this section, let $\CC$ be a 2-CY triangulated $k$-category
and let $T$ be a cluster-tilting object in $\CC$. Let $\Lambda:=\End_{\CC}(T)^{\rm op}$.
For any object $X\in\CC$, we take a decomposition $X=X'\oplus X''$ where $X''$ is a
maximal direct summand of $X$ which belongs to $\add T[1]$ and define a pair by
\[\widetilde{X}_\Lambda:=(\functor{X'},\functor{X''[-1]}),\]
where $\functor{(-)}$ is an equivalence $\Hom_{\CC}(T,-):\CC/[T[1]]\to\mod\Lambda$ given in \eqref{bar functor}.

We have the following interpretation of $E$-invariants.

\begin{proposition}\label{E=Ext^1}
We have $E_\Lambda(\widetilde{X}_\Lambda,\widetilde{Y}_\Lambda)=\dim_k\Ext^1_{\CC}(X,Y)$ for any $X,Y\in\CC$.
\end{proposition}

\begin{proof}
This is immediate from Proposition \ref{thm3.3} and our definition of $E$-invariants.
\end{proof}

Now let $T'$ be a cluster-tilting mutation of $T$.
Then we refer to the 2-CY-tilted algebras $\Lambda=\End_{\CC}(T)^{\op}$ and 
$\Lambda'=\End_{\CC}(T')^{\op}$ as {\it neighbouring}\index{neighbouring 2-CY-tilted algebras} 2-CY-tilted algebras.
We define a pair $\widetilde{X}_{\Lambda'}$ for $\Lambda'$ in a similar way to
$\widetilde{X}_\Lambda$ by using the equivalence $\Hom_{\CC}(T',-):\CC/[T'[1]]\to\mod\Lambda'$.

By our approach to the $E$-invariant, the following is now a direct consequence.

\begin{theorem}\label{thm4.3}
With the above notation, let $M$ and $N$ be objects in $\CC$.
Then $E_{\Lambda}(\widetilde{M}_\Lambda,\widetilde{N}_\Lambda)=
E_{\Lambda'}(\widetilde{M}_{\Lambda'},\widetilde{N}_{\Lambda'})$.
\end{theorem}

\begin{proof}
This is clear from Proposition \ref{E=Ext^1} since both sides are equal to $\dim_k\Ext^1_{\CC}(M,N)$.
\end{proof}

In particular, $\widetilde{M}_\Lambda$ is $\tau$-rigid if and only if $\widetilde{M}_{\Lambda'}$ is $\tau$-rigid.

This result is analogous to the corresponding result for (neighbouring) Jacobian 
algebras proved in \cite{DWZ}, in a larger generality. 
It is however not clear whether the two concepts of neighbouring algebras coincide
for finite dimensional neighbouring Jacobian algebras.
See \cite{BIRS} for more information.

\section{Examples}

In this section we illustrate some of our work with easy examples.

\begin{example}
Let $\Lambda$ be a local finite dimensional $k$-algebra.
Then we have $\sttilt\Lambda=\{\Lambda,0\}$ since the condition $\Hom_\Lambda(M,\tau M)=0$ implies either $M=0$ or $\tau M=0$ (i.e. $M$ is projective).
We have $\Q(\sttilt\Lambda)=(\xymatrix{\Lambda\ar[r]&0})$,
$\Q(\ftors\Lambda)=(\xymatrix{\mod\Lambda\ar[r]&0})$ and
$\Q(\twosilt\Lambda)=(\xymatrix{\Lambda\ar[r]&\Lambda[1]})$.
\end{example}

\begin{example}
Let $\Lambda$ be a finite dimensional $k$-algebra given by the quiver
 $\xymatrix{1\ar@<.3mm>[r]^a&2\ar@<.3mm>[l]^a}$ with relations $a^2=0$.
Then $\Q(\sttilt\Lambda)$, $\Q(\ftors\Lambda)$ and $\Q(\twosilt\Lambda)$ are the following:
\begin{eqnarray*}&\xymatrix{
{\begin{smallmatrix}1\\ 2\end{smallmatrix}}\oplus{\begin{smallmatrix}2\\ 1\end{smallmatrix}}\ar[r]\ar[d]&{\begin{smallmatrix}1\\ 2\end{smallmatrix}}\oplus 1\ar[r]&1\ar[d]\\
2\oplus{\begin{smallmatrix}2\\ 1\end{smallmatrix}}\ar[r]&2\ar[r]&0
}&\\
&\xymatrix{
\mod\Lambda\ar[r]\ar[d]&\add({\begin{smallmatrix}1\\ 2\end{smallmatrix}}\oplus 1)\ar[r]&\add 1\ar[d]\\
\add(2\oplus{\begin{smallmatrix}2\\ 1\end{smallmatrix}})\ar[r]&\add 2\ar[r]&0
}&\\
&\xymatrix{
\Lambda\ar[r]\ar[d]&
\left[{\begin{smallmatrix}2\\ 1\end{smallmatrix}}\xrightarrow{[a\ 0]}{\begin{smallmatrix}1\\ 2\end{smallmatrix}}\oplus{\begin{smallmatrix}1\\ 2\end{smallmatrix}}\right]\ar[r]&
\left[{\begin{smallmatrix}2\\ 1\end{smallmatrix}}\oplus{\begin{smallmatrix}2\\ 1\end{smallmatrix}}\xrightarrow{[a\ 0]}{\begin{smallmatrix}1\\ 2\end{smallmatrix}}\right]\ar[d]\\
\left[{\begin{smallmatrix}1\\ 2\end{smallmatrix}}\xrightarrow{[a\ 0]}{\begin{smallmatrix}2\\ 1\end{smallmatrix}}\oplus{\begin{smallmatrix}2\\ 1\end{smallmatrix}}\right]\ar[r]&
\left[{\begin{smallmatrix}1\\ 2\end{smallmatrix}}\oplus{\begin{smallmatrix}1\\ 2\end{smallmatrix}}\xrightarrow{[a\ 0]}{\begin{smallmatrix}2\\ 1\end{smallmatrix}}\right]\ar[r]&
\Lambda[1]
}&\end{eqnarray*}
\end{example}

\begin{example}
Let $\Lambda$ be a finite dimensional $k$-algebra given by the quiver
 $\xymatrix@C=.5em@R=.5em{&2\ar[rd]^a\\ 1\ar[ur]^a&&3\ar[ll]^a}$ with relations $a^2=0$.
Then $\Lambda$ is a cluster-tilted algebra of type $A_3$, and 
there are 14 elements in $\ctilt\CC$ for the cluster category $\CC$ of type $A_3$.
By our bijections, we know that there are 14 elements in each set
$\sttilt\Lambda$, $\ftors\Lambda$ and $\twosilt\Lambda$.

\[\xymatrix@R=1em{
&{\begin{smallmatrix}1\\ 2\end{smallmatrix}}\oplus{\begin{smallmatrix}2\\ 3\end{smallmatrix}}\oplus 2\ar[rr]\ar[rd]&&
{\begin{smallmatrix}2\\ 3\end{smallmatrix}}\oplus2\ar[rd]\\
&&{\begin{smallmatrix}1\\ 2\end{smallmatrix}}\oplus2\ar[rr]\ar[rd]&&2\ar[rdd]\\
{\begin{smallmatrix}1\\ 2\end{smallmatrix}}\oplus{\begin{smallmatrix}2\\ 3\end{smallmatrix}}\oplus{\begin{smallmatrix}3\\ 1\end{smallmatrix}}\ar[ruu]\ar[r]\ar[rdd]&
{\begin{smallmatrix}1\\ 2\end{smallmatrix}}\oplus1\oplus{\begin{smallmatrix}3\\ 1\end{smallmatrix}}\ar[rr]\ar[rd]&&
{\begin{smallmatrix}1\\ 2\end{smallmatrix}}\oplus1\ar[rd]\\
&&{\begin{smallmatrix}3\\ 1\end{smallmatrix}}\oplus1\ar[rr]\ar[rd]&&1\ar[r]&0\\
&3\oplus{\begin{smallmatrix}2\\ 3\end{smallmatrix}}\oplus{\begin{smallmatrix}3\\ 1\end{smallmatrix}}\ar[rr]\ar[rd]&&
{\begin{smallmatrix}3\\ 1\end{smallmatrix}}\oplus3\ar[rd]\\
&&{\begin{smallmatrix}2\\ 3\end{smallmatrix}}\oplus3\ar[rr]\ar[ruuuuu]&&3\ar[ruu]
}\]
\end{example}

\begin{example}
Let $\Lambda=kQ/\langle\beta\alpha\rangle$, where $Q$ is the quiver
$1\stackrel{\alpha}{\longrightarrow}2\stackrel{\beta}{\longrightarrow}3$. Then $T=S_1\oplus P_1\oplus P_3$ is a $\tau$-tilting
module which is not a tilting module. Here $S_i$ denotes the simple $\Lambda$-module associated with the vertex $i$, and $P_i$ denotes the corresponding indecomposable projective $\Lambda$-module.

In this case there are 12 basic support $\tau$-tilting $\Lambda$-modules, and $\Q(\sttilt\Lambda)$ is the following.
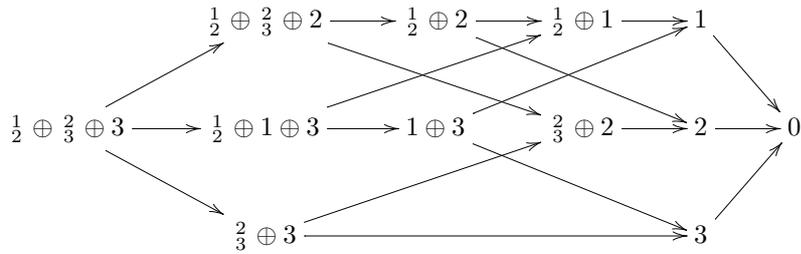
\[\xymatrix{
&{\begin{smallmatrix}1\\ 2\end{smallmatrix}}\oplus{\begin{smallmatrix}2\\ 3\end{smallmatrix}}\oplus 2\ar[r]\ar[rrd]&
{\begin{smallmatrix}1\\ 2\end{smallmatrix}}\oplus 2\ar[r]\ar[rrd]&
{\begin{smallmatrix}1\\ 2\end{smallmatrix}}\oplus 1\ar[r]&1\ar[rd]\\
{\begin{smallmatrix}1\\ 2\end{smallmatrix}}\oplus{\begin{smallmatrix}2\\ 3\end{smallmatrix}}\oplus 3\ar[ur]\ar[r]\ar[dr]&
{\begin{smallmatrix}1\\ 2\end{smallmatrix}}\oplus 1\oplus 3\ar[rru]\ar[r]&
1\oplus 3\ar[rru]\ar[rrd]&
{\begin{smallmatrix}2\\ 3\end{smallmatrix}}\oplus 2\ar[r]&
2\ar[r]&0\\
&{\begin{smallmatrix}2\\ 3\end{smallmatrix}}\oplus 3\ar[rrr]\ar[urr]&&&3\ar[ru]
}\]
\end{example}

We refer to \cite{Ad,J,Miz,Z} for more examples of support $\tau$-tilting modules.

\printindex


\begin{thebibliography}{4} 
\bibitem[Ab]{Ab} H.~Abe, \emph{Tilting modules arising from two-term tilting complexes},
arXiv:~1104.0627.
\bibitem[Ad]{Ad} T.~Adachi, \emph{$\tau$-tilting modules over Nakayama algebras}, in preparation.
\bibitem[Ai]{Ai} T.~Aihara, \emph{Tilting-connected symmetric algebras},
to appear in Algebr. Represent. Theory, arXiv:~1012.3265.
\bibitem[AI]{AI} T.~Aihara, O.~Iyama, \emph{Silting mutation in triangulated categories},
J. Lond. Math. Soc. 85 (2012), no. 3, 633--668.
\bibitem[A]{A} C. Amiot, {\it Cluster categories for algebras of global dimension 2 and quiver with potential}, Annales de l'Institut Fourier, vol. 59, no.6 (2009) 2525--2590.
\bibitem[ASS]{ASS} I. Assem, D. Simson and A. Skowronski, {\it Elements of the representation theory of associative algebras}, Vol.65, Cambridge Univ. Press, Cambridge 2006.
\bibitem[APR]{APR} M. Auslander, M.I. Platzeck and I. Reiten, {\it Coxeter functions without diagrams}, Trans. Amer. Math. Soc. 250 (1979) 1--12.
\bibitem[AR1]{AR1} M. Auslander and I. Reiten, {\it Representation theory of artin algebras III: Almost split sequences}, Comm. in Alg. 3 (1975), 239--294.
\bibitem[AR2]{AR2} M. Auslander and I. Reiten, {\it Representation theory of artin algebras V: Methods for computing almost split sequences and irreducible morphisms}, Comm. in Alg. 5(5) (1977), 519--554.
\bibitem[AR3]{AR3} M. Auslander and I. Reiten, {\it Modules determined by their composition factors}, Ill. J. of Math 29 (1985), 280--301.
\bibitem[AR4]{AR4} M. Auslander and I. Reiten, {\it Applications of contravariantly finite subcategories}, Adv. Math. 86 (1991), no. 1, 111--152.
\bibitem[ARS]{ARS} M. Auslander, I. Reiten and S. O. Smal\o, {\it Representation theory of artin algebras}, Cambridge studies in advanced mathematics 36, Cambridge Univ. Press 1995.
\bibitem[AS]{AS} M. Auslander and S. O. Smal\o, {\it Almost split sequences in subcategories}, J. Algebra 69 (1981) 426--454. Addendum; J. Algebra 71 (1981), 592--594.
\bibitem[BGP]{BGP} I. N. Bernstein, I. M. Gelfand and V. A. Ponomarev, {\it Coxeter functors and Gabriel's theorem}, Russ. Math. Surv. 28 (1973), 17--32.
\bibitem[B]{B} K. Bongartz, {\it Tilted Algebras}, Proc. ICRA III (Puebla 1980), Lecture Notes in Math. No. 903, Springer-Verlag 1981, 26--38.
\bibitem[BB]{BB} S. Brenner and M. C. R. Butler, {\it Generalization of the Bernstein-Gelfand-Ponomarev reflection functors}, Lecture Notes in Math. 839, Springer-Verlag (1980), 103--169.
\bibitem[BIRS]{BIRS} A. Buan, O. Iyama, I. Reiten and D. Smith, {\it Mutation of cluster-tilting objects and potentials}, Amer. J. Math. 133 (2011), no. 4, 835--887
\bibitem[BMRRT]{BMRRT} A. B. Buan, R. Marsh, M. Reineke, I. Reiten and G. Todorov, {\it Tilting theory and cluster combinatorics}, Adv. Math. 204 (2006), no. 2, 572--618.
\bibitem[BMR1]{BMR1} A.B. Buan, R. Marsh and I. Reiten, {\it Cluster-tilted algebras}, Trans. Amer. Math. Soc. 359(2007), no. 1, 323--332.
\bibitem[BMR2]{BMR2} A.B. Buan, R. Marsh and I. Reiten, {\it Denominators of cluster variables}, J. London Math. Soc. (2) 79 (2006) no.3, 589--611.
\bibitem[BMRT]{BMRT} A. B. Buan, R. Marsh, I. Reiten and G. Todorov, {\it Clusters and seeds in acyclic cluster algebras}, Proc. Amer. Math. Soc. 135 (2007), no. 10, 3049--3060, with an appendix coauthored in addition by P. Caldero and B. Keller.
\bibitem[BRT]{BRT} A. B. Buan, I. Reiten, H. Thomas, {\it Three kinds of mutation}, J. Algebra 339 (2011), 97--113.
\bibitem[CK]{CK} P. Caldero and B. Keller, {\it From triangulated categories to cluster algebras II}, Ann. Sci Ecole Norm. Sup. (4), 39 (2006), no. 6, 983--1009.
\bibitem[CLS]{CLS} G. Cerulli Irelli, D. Labardini-Fragoso, J. Schr\"oer, {\it Caldero-Chapoton algebras}, arXiv:1208.3310.
\bibitem[DK]{DK} R. Dehy and B. Keller, {\it On the combinatorics of rigid objects in 2-Calabi-Yau categories}, Int. Math. Res. Not. IMRN 2008, no. 11, Art. ID rnn029.
\bibitem[DF]{DF} H. Derksen, J. Fei, {\it General Presentations of Algebras}, arXiv:0911.4913.
\bibitem[DWZ]{DWZ} H. Derksen, J. Weyman and A. Zelevinsky, {\it Quivers with potentials and their representations II: Applications to cluster algebras}, J. Amer. Math. Soc.23 (2010), no.3, 749--790.
\bibitem[FZ]{FZ} S. Fomin and A. Zelevinsky, {\it Cluster algebras IV: Coefficients}, Compos. Math. 143 (2007), no.3, 112--164.
\bibitem[FL]{FL} C. Fu and P. Liu, {\it Lifting to cluster-tilting objects in 2-Calabi-Yau triangulated categories}, Comm. Algebra 37 (2009), no. 7, 2410--2418.
\bibitem[Ha]{Ha} D. Happel, \emph{Triangulated categories in the representation theory of finite-dimensional algebras},
London Mathematical Society Lecture Note Series, 119. Cambridge University Press, Cambridge, 1988.
\bibitem[HR]{HR} D. Happel and C. M. Ringel, {\it Tilted algebras}, Trans. Amer.  Math. Soc. 274 (1982), 399--443.
\bibitem[HU1]{HU1} D. Happel and L. Unger, {\it Almost complete tilting modules}, Proc. Amer. Math. Soc. 107 (1989), no. 3, 603--610
\bibitem[HU2]{HU2} D. Happel and L. Unger, {\it On a partial order of tilting modules}, Algebr. Represent. Theory 8 (2005), no. 2, 147--156.
\bibitem[Ho]{Ho} M. Hoshino, \emph{Tilting modules and torsion theories},
Bull. London Math. Soc. 14 (1982), no. 4, 334--336.
\bibitem[HKM]{HKM} M. Hoshino, Y. Kato, J. Miyachi, {\it On $t$-structures and torsion theories induced by compact objects}, J. Pure Appl. Algebra 167 (2002), no. 1, 15--35.
\bibitem[IT]{IT} C. Ingalls and H. Thomas, {\it Noncrossing partitions and representations of quivers}, Compos. Math. 145 (2009), no. 6, 1533--1562.
\bibitem[IY]{IY} O. Iyama and Y. Yoshino, {\it Mutations in triangulated categories and rigid Cohen-Macaulay modules}, Inv. Math. 172 (2008), 117--168.
\bibitem[J]{J} G. Jasso, {\it Reduction of $\tau$-tilting modules and torsion pairs}, arXiv:~1302.2709.
\bibitem[KR]{KR} B. Keller and I. Reiten, {\it Cluster-tilted algebras are Gorenstein and stably Calabi-Yau}, Adv. Math. 211 (2007), 123--151.
\bibitem[KV]{KV} B.~Keller, D.~Vossieck, \emph{Aisles in derived categories},
Deuxieme Contact Franco-Belge en Algebre (Faulx-les-Thombes, 1987).
Bull.\ Soc.\ Math.\ Belg.\ \textbf{40} (1988), 239--253.
\bibitem[KY]{KY} S. K\"onig, D. Yang, {\it Silting objects, simple-minded collections, $t$-structures and co-$t$-structures for finite-dimensional algebras}, arXiv:1203.5657.
\bibitem[Miy]{Miy} Y. Miyashita, {\it Tilting modules of finite projective dimension}, Math. Z. 193 (1986), no. 1, 113--146.
\bibitem[Miz]{Miz} Y.~Mizuno, \emph{Classifying $\tau$-tilting modules over preprojective algebras of Dynkin type}, arXiv:~1304.0667.
\bibitem[P]{P} Y. Palu, {\it Cluster characters for 2-Calabi-Yau triangulated categories}, Ann. Inst. Fourier (Grenoble) 58 (2008), no. 6, 2221--2248.
\bibitem[RS]{RS} C. Riedtmann and A. Schofield, {\it On a simplicial complex associated with tilting modules}, Comment. Math. Helv. 66 (1991), no. 1, 70--78.
\bibitem[Ric]{Ric} J. Rickard, {\it Morita theory for derived categories}, J. London Math. Soc. (2) 39 (1989), no. 3, 436--456.
\bibitem[Ri]{Ri} C. M. Ringel, {\it Some remarks concerning tilting modules and tilted algebras}, Origin. Relevance. Future, Handbook of Tilting Theory, LMS Lecture Note Series 332 (2007), 49--104.
\bibitem[Sk]{Sk} A. Skowronski, {\it Regular Auslander-Reiten components containing directing modules}, Proc. Amer. Math. Soc. 120, (1994), 19--26.
\bibitem[Sma]{Sma} S. O. Smal\o, {\it Torsion theory and tilting modules}, Bull. London Math. Soc. 16(1984), 518--522.
\bibitem[Smi]{Smi} D. Smith, {\it On tilting modules over cluster-tilted algebras}, Illinois J. Math. 52 (2008), no.4, 1223--1247.
\bibitem[U]{U} L. Unger, {\it Schur modules over wild, finite-dimensional path algebras with three simple modules}, J. Pure Appl. Alg. 64. no.2 (1990) 205--222.
\bibitem[Z]{Z} X.~Zhang, {\it $\tau$-rigid modules for algebras with radical square zero}, arXiv:~1211.5622.
\bibitem[ZZ]{ZZ} Y. Zhou and B. Zhu, {\it Maximal rigid subcategories in 2-Calabi-Yau triangulated categories}, J. Algebra 348 (2011), no. 1, 49--60.
\end{thebibliography}
\end{document}